\documentclass[10pt,a4paper]{article}

\usepackage{geometry}

\usepackage{color,array}

\usepackage{enumitem}
\usepackage[utf8]{inputenc} 

\usepackage{amsmath,amssymb,amsthm}
\usepackage{booktabs}
\usepackage{adjustbox}
\usepackage{multirow}
\usepackage{enumitem}
\usepackage{microtype}
\usepackage{cleveref}

\DisableLigatures[<,>]{encoding=T1,family=tt*}

\usepackage[english]{babel}  
\usepackage{lmodern}             
\usepackage{mathtools}
\usepackage{subcaption}
\allowdisplaybreaks

\usepackage{color}
\definecolor{darkred}{RGB}{100,0,0}
\definecolor{darkgreen}{RGB}{0,100,0}
\definecolor{darkblue}{RGB}{0,0,150}
\definecolor{citecol}{RGB}{30,80,150}
\definecolor{purple}{RGB}{160,32,240}

\usepackage{url}

\newcommand{\E}{{\mathbb E}}

\renewcommand{\P}{{\mathbf P}}
\newcommand{\Q}{{\mathbf Q}}
\newcommand{\R}{{\mathbb R}}
\newcommand{\cR}{{\mathcal R}}

\newcommand{\ind}{{\mathbf 1}}
\newcommand{\vol}{ \mathrm{vol}}

\renewcommand\path{\mathfrak{P}}

\newcommand{\cL}{\mathcal{L}} 
 
\newcommand{\B}{\mathcal{B}} 

\newcommand\dif{\mathrm{ d}}

\newcommand\cA{\mathcal{A}}
\newcommand\cB{\mathcal{B}}
\newcommand\cC{\mathcal{C}}
\newcommand\cE{\mathcal{E}}
\newcommand\cF{\mathcal{F}}

\newcommand\cJ{\mathcal{J}}
\newcommand\cM{\mathcal{M}}
\newcommand\cW{\mathcal{W}}
\newcommand\W{\mathcal{W}}
\newcommand\cP{\mathcal{P}}

\newcommand\dR{\mathbb{R}}

\newcommand\dd{\mathrm{d}}

\newcommand{\ABS}[1]{{{\left| #1 \right|}}} 
\newcommand{\BRA}[1]{{{\left\{#1\right\}}}} 
\newcommand{\SCA}[1]{{{\left\langle #1\right\rangle}}} 
\newcommand{\NRM}[1]{{{\left\| #1\right\|}}} 
\newcommand{\PAR}[1]{{{\left(#1\right)}}} 
\newcommand{\SBRA}[1]{{{\left[#1\right]}}} 
\renewcommand{\leq}{\leqslant}
\renewcommand{\geq}{\geqslant}
\newcommand{\eps}{{\varepsilon}}

\renewcommand{\theequation}{\arabic{section}.\arabic{equation}}
\let \ssection=\section
\renewcommand{\section}{\setcounter{equation}{0}\ssection}

\newtheorem{theorem}{Theorem}[section] 
\newtheorem{definition}[theorem]{Definition} 
\newtheorem{remark}[theorem]{Remark}

\newtheorem{proposition}[theorem]{Proposition}

\newtheorem{corollary}[theorem]{Corollary}
\newtheorem{lemma}[theorem]{Lemma}
\newtheorem{assump}[theorem]{Assumption}

\begin{document}
		\title{Measure estimation on a manifold explored by a diffusion process}
	\author{Vincent Divol, Hélène Guérin, Dinh-Toan Nguyen, Viet Chi Tran}

	\date{\today}
	
	\maketitle

\begin{abstract}
From the observation of a diffusion path $(X_t)_{t\in [0,T]}$ on a compact connected  $d$-dimensional manifold $\cM$ without boundary, we consider the problem of estimating the stationary measure $\mu$ of the process. Wang and Zhu (2023) showed that for the Wasserstein metric $\W_2$ and for $d\ge 5$, the convergence rate of $T^{-1/(d-2)}$ is attained by the occupation measure of the path  $(X_t)_{t\in [0,T]}$  when $(X_t)_{t\in [0,T]}$ is a Langevin diffusion. We extend their result in several directions. First, we show that the rate of convergence holds for a large class of diffusion paths, whose generators are uniformly elliptic. Second, the regularity of the density $p$ of the stationary measure $\mu$ with respect to the volume measure of $\cM$ can be leveraged to obtain faster estimators: when $p$ belongs to a Sobolev space of order $\ell\geq 2$, smoothing the occupation measure by convolution with a kernel yields an estimator whose rate of convergence is of order $T^{-(\ell+1)/(2\ell+d-2)}$. We further show that this rate is the minimax rate of estimation for this problem.
\end{abstract}

\noindent \textit{MSC2020 subject classifications:} 60F17, 05C81, 62G07. \\


\noindent \textit{Acknowledgements:} We would like to thank the anonymous Referees for their valuable comments and questions that help improve our work. We also thank Frédéric Rochon for his advice on the differential geometry literature, which helped us to tackle some concepts more efficiently.
The research of H.G. and D.T.N. is part of the Mathematics for Public Health program funded by the Natural Sciences and Engineering Research Council of Canada (NSERC). H.G. is also supported by NSERC discovery grant (RGPIN-2020-07239). D.T.N. and V.C.T. are supported by Labex B\'ezout (ANR-10-LABX-58), GdR GeoSto 3477 and by the European Union (ERC-AdG SINGER-101054787).

 
\section{Introduction}\label{sec:intro}

The manifold hypothesis has become ubiquitous in the modern machine learning landscape, where it is commonly used to explain the efficiency of nonparametric methods in high-dimensional statistical models \cite{brown2022union}. This paradigm has motivated statisticians to study inference problems under manifold constraints \cite{niyogi2008finding, genovese2012minimax, aamarilevrard2018, aamarilevrard2019, divol2021minimax, puchkin2022structure}. Given $n$ i.i.d.~samples from a distribution $\mu$ supported by a $d$-dimensional manifold compact $\cM$, the task of estimating   either $\mu$ or geometric quantities related to $\cM$ naturally arises. The picture is now well-understood. For example, minimax rates  are known for the estimation of $\cM$ with respect to the Hausdorff distance \cite{genovese2012manifold}, for tangent space estimation \cite{tyagi2013tangent, aamarilevrard2019}, and for curvature estimation \cite{aamari2019estimating, aamarilevrard2019, berenfeld2022estimating, aamari2023optimal}.  
The estimation of the measure $\mu$ has been tackled in a pointwise manner \cite{berenfeld2021density}, with respect to the Wasserstein distance \cite{Divol2022, stephanovitch2023wasserstein}, or with respect to more general adversarial losses \cite{tang2023minimax}. Once again, minimax rates of estimation are known and are typically achieved by kernel-like estimators. Yet, the literature is far less abundant when we leave the i.i.d. setting. 

However, a framework in which the data is generated through an exploration process is also natural. This setting is especially relevant in scenarios  where the manifold is seen as the continuum limit of a large graph, the latter being explored by a random walk. We can think for instance of the famous PageRank algorithm \cite{Pagebrinmotwaniwinograd}, {but also of Respondent-Driven Sampling methods to discover hard-to-reach populations such as sexual communities or drug-users \cite{heckathorn,cousiendhersintranvo,tranvo,vo_esaimps}}. At the limit, this random walk converges to a continuous time diffusion exploring the manifold. 
Formally, we will consider that we have access to a sample path $(X_t)_{t\in [0,T]}$ on $[0,T]$ of a diffusion process on a submanifold $\cM\subseteq \R^m$, that is generated by a uniformly elliptic $\cC^2$-differential operator $\cA$, essentially self-adjoint with respect to some invariant measure $\mu$. Our goal is to propose reconstruction methods for the measure $\mu$ based on the observation of the sample path $(X_t)_{t\in [0,T]}$. {Applications of these results could include asymptotic equivalence \cite{dalalyan2006asymptotic}, bias correction in sampling schemes based on random walks \cite{tranvo} or manifold regression \cite{rosarousseau}, for example}.\\

The general framework we work in encapsulates in particular operators of the form  $\mathcal{A}_{pq}$ given for any test function $f$ of class $\cC^2$ on $\cM$ by
\begin{equation}\label{eq:A}
\cA_{pq}(f):=q\Delta f+\SCA{q\nabla \ln (p q),\nabla f},
\end{equation}
where $p ,\ q\in \cC^2$ are two positive functions, with $p$ being the density of the measure $\mu$ with respect to the volume measure $\dd x$ on $\cM$ 
and  $\nabla$ and $\Delta$ denote respectively the gradient and the Laplace-Beltrami operator on $\cM$ (see e.g. \cite{Wang2014}). 
When we take $q=\frac{p}{2}$, we recover the generator 
\[\frac{p}{2}\Delta f+\SCA{\nabla p,\nabla f}\]
studied in \cite{Calder2022,Gine2006,guerinnguyentran}. When $q= 1$, we recover a Langevin diffusion, whose generator $\cL$ is  defined for any test function $f$ of class $\cC^2$ on $\cM$ by
\begin{align}\label{eq:L}
 \cL(f):= & \Delta f+\SCA{\nabla \ln p, \nabla f}
 =  \Delta f +\SCA{\frac{\nabla p}{p}, \nabla f}.
 \end{align}
Processes with this kind of generators can be obtained as the limits of random walks (without and with renormalization) visiting points sampled independently with identical distribution (i.i.d.) $\mu(\dd x)=p(x)\dd x$ on $\cM$, see e.g. \cite{Calder2022,Gine2006,guerinnguyentran}.

In $\R^m$, the question of estimating the invariant measure of a diffusion has been treated abundantly, see \cite{Dalalyan,durmusmoulines,pagespanloup2012,Robertstweedie} (notice that the problem could also be studied with the different point of view of non-parametric estimation for diffusion processes, see e.g. \cite{dalalyan2007asymptotic}).
For manifold-valued data, the problem of reconstructing the stationary measure $\mu$ from a sample path was first addressed by Wang and Zhu \cite{Wang2023} for the generator \eqref{eq:L}. They consider the occupation measure $\mu_T$ of the process, defined for every bounded measurable test function $f$ by
\begin{equation}
	\int_{\cM} f(x)    \mu_T(\dd x)=\frac{1}{T} \int_0^T f(X_s)\ \dd s. \label{def:occ_measure}
\end{equation}
Let $\cP(\cM)$ be the space of probability measures on the compact connected $d$-dimensional Riemannian manifold $\cM$. 
We introduce the $2$-Wasserstein distance on $\cP(\cM)$, defined by
\[
\W_2(\mu_1,\mu_2):=\inf_{\pi\in \cC(\mu_1,\mu_2)}\PAR{\int_{\cM\times \cM}\rho (x,y)^2\pi(\dd x,\dd y)}^{1/2},
\]
where $\rho$ is the geodesic distance on $\cM$ and $\cC(\mu_1,\mu_2)$ is the set of measures on $\cM\times \cM$ with first marginal $\mu_1$ and second marginal $\mu_2$. Wang and Zhu \cite[Theorem 1.2]{Wang2023} showed that for the process with generator $\cL$, 
\begin{equation}\label{eq:rate_WZ}
  \E_x\SBRA{\W_2^2(\mu_T,\mu)} \lesssim \begin{cases}
 T^{-1} & \mbox{ when }d\le 3 \\
 T^{-1} \ln(1+T) & \mbox{ when }d=4\\
 T^{-\frac{2}{d-2}} & \mbox{ when }d\ge 5,
\end{cases}  
\end{equation}
where $\E_x$ stands for the expectation taken from the diffusion process starting at $x\in \cM$. 
 As noticed by Divol \cite{Divol2022}, in the context of i.i.d. random variables $X_1,\dots X_n$ sampled from $\mu$ on $\cM$, the rate of convergence can be increased by smoothing the empirical measure. Our purpose here is to  extend this result beyond the i.i.d. setting, by studying the convergence properties of an estimator $\widehat{\mu}_{T,h}$ of $\mu$, obtained by  smoothing the occupation measure $\mu_T$ with a kernel $K$ of bandwidth $h>0$. When $d\ge 5$ and for an appropriate choice of $h$, we obtain the rate of convergence
\begin{equation}
\E_x\SBRA{\W_2^2(\widehat{\mu}_{T,h},\mu)} \lesssim T^{-\frac{2\ell+2}{2\ell+d-2}},
\end{equation}
where $\mu$ has a density of regularity $\ell \ge 2$. The above rate does not only hold for the Langevin diffusion with generator $\cL$, but for all diffusion paths $(X_t)_{t\in [0,T]}$ whose generator $\cA$ is a uniformly elliptic $\cC^2$-differential operator, essentially self-adjoint with respect to $\mu$. Furthermore, 
we will show that these rates cannot be improved by providing minimax rates of convergence for this problem.

In Section \ref{sec:main-results}, we define the estimator $\widehat{\mu}_{T,h}$ and enounce precisely our main result. In Section \ref{sec:preliminaries}, we review some useful notions of Riemannian geometry. In Section \ref{sec:proof-mainA}, we start the proof by treating the stationary case, i.e., when the initial measure of the SDE is $\mu$, and then give a generalization for a general initial measure.

\medskip

\textbf{Notation.} Throughout the paper, we fix  a smooth compact $d$-dimensional connected submanifold $\cM$ of $\mathbb{R}^m$,  without boundary, and embedded with the Riemannian structure induced by the ambient space $\mathbb{R}^m$. The volume measure on $\cM$ is denoted by $\dd x$. Without loss of generalization, we assume that $\vol(\cM)=1$. Unless stated otherwise, quantities $c_0,c_1,\dots$ are  constants that are only allowed to depend on the manifold $\cM$. We write $c_a$ for a constant depending on an additional parameter $a$. The geodesic distance on $\cM$ is denoted by $\rho$, and $\cB(x,r)$ is the geodesic open ball centered at $x\in \cM$ of radius $r\ge 0$. We also let $\cB_{\R^d}(u,r)$ be the open ball centered at $u\in \R^d$ of radius $r$.

For $\mu$ a probability measure, we will denote by $L^2(\mu)$ the space of real-valued measurable functions $f$ on $\cM$ such that $\int\ABS{f}^2\dd \mu<\infty$. More generally, for $p\ge 1$,  we let $L^p(\mu)$ denote the space of $L^p$ functions with respect to $\mu$,  with $\|\cdot\|_{L^p(\mu)}$ the corresponding norm. For $k\ge 0$, we denote by $\cC^k(\cM)$ the space of $k$-times continuously differentiable real-valued functions defined on $\cM$, endowed with the norm
 \begin{align}\label{eq:holder-norm}
 \|f\|_{\cC^k(\cM)}:= \sup_{0\le i\le k}\, \sup_{x\in \cM} \|\nabla^i f(x)\|, 
 \end{align}
where $\nabla^i f(x)$ is the $i$th iterated covariant derivative of $f$ at $x$. By abuse of notation, we use the same notation for the uniform norm $\NRM{.}_\infty$ on $\dR_+$, on $\cM$, and on $\dR^d$.

\section{Main results}\label{sec:main-results}

We consider the framework described at the beginning of the paper. 
 Let $ \mu(\dd x)=p(x)\dd x$ be a probability measure on $\cM$ with a positive density $p$ of class $\cC^2$,  and $(X_t)_{t\ge 0}$ be  a diffusion on $\cM$ with generator $\cA$.
\begin{assump}\label{assump:A} 
$\cA: \cC^\infty(\cM) \rightarrow \cC^2(\cM)$ is a
uniformly elliptic $\mathcal{C}^2$-differential operator of second order on $\cM$, symmetric with respect to the measure $\mu$. \end{assump}

Notice that the operators $\cA_{pq}$ satisfy Assumption \ref{assump:A} as soon as $p$ and $q$ are of class $\cC^2$ (see e.g. \cite[Exercise 3.11, p. 69]{Grigoryan2009} for smooth functions, which can be extended to the case where $p$ and $q$ are of class $\cC^2$). \\
The diffusion for $\cA$ satisfying  Assumption \ref{assump:A} admits $\mu$ as stationary measure. In the sequel, this generator is extended, as is usually done, to $L^2(\mu)$ functions. 
We introduce a concept closely associated with second-order differential operators, known as the \textit{carré du champ} $\Gamma(f,g)$ for the operator $\mathcal{A}$: 
\begin{equation}\label{eq:carre-du-champ}
\Gamma(f,g) = \frac{1}{2}\left( \mathcal{A}(fg) - f\mathcal{A}(g) - g\mathcal{A}(f) \right)
\end{equation}
Given that $\mathcal{A}$ is symmetric with respect to $\mu$, for any smooth functions $f$ and $g$, it follows that:
\begin{equation}\label{eq:green-formula-generalA}
\int_\cM \Gamma(f,g) \mathrm{d}\mu = - \int_\cM f\mathcal{A}(g) \mathrm{d}\mu =-\int_\cM \mathcal{A}(f)g \mathrm{d}\mu.
\end{equation}
Since $\cM$ is compact, from the regularity of $p$, there exist $p_{\min}, \, p_{\max}>0$ such that, 
\begin{equation} \label{eq:pmin-pmax}
\forall x\in\cM,\quad p_{\min}\le p(x)\le p_{\max}.  
\end{equation}
Furthermore, the uniform ellipticity, the continuity of $\mathcal{A}$ and the compactness of $\mathcal{M}$ imply that 
there exist positive constants $\kappa_{\min}$, $\kappa_{\max}$ such that for all functions $f$,
\begin{equation}\label{eq:A-def-kappa}
\kappa_{\min}|\nabla f|^2 \le \Gamma(f,f) \le \kappa_{\max} |\nabla f|^2.
\end{equation}

We denote by $\E_{\mu_0}$ and $\P_{\mu_0}$ the expectation and probability taken for the diffusion process with initial distribution $\mu_0$. When $\mu_0$ is a Dirac measure $\delta_x$, with $x\in\cM$, we will simply write $\E_x$ and $\P_x$.

Let $T>0$ be a time horizon, and assume that we observe the diffusion $(X_t)_{t\ge 0}$ on the time window $[0,T]$. We consider the occupation measure $\mu_T$ on $[0,T]$ of the diffusion $(X_t)_{t\ge 0}$ as the positive measure defined by \eqref{def:occ_measure}. 
As explained in the introduction, the occupation measure can be seen as a first naive estimator of the measure $\mu$, that will be improved upon by convolution with a kernel $K_h$.

    Let $K: \mathbb{R}_+\rightarrow \mathbb{R}$ be a (signed) Lipschitz-continuous function  supported in $[0,1]$, with $\int_{\R^d} K(\|u\|)\dd u=1$. 
We define for $h>0$ small enough (see~\Cref{lem:properties-Kh} $(ii)$) and $(x,y)\in\cM^2$,
\begin{equation}\label{eq:def-K_h}
K_h(x,y):=\frac{1}{\eta_h(x)}K\PAR{\frac{\NRM{x-y}}{h}},
\end{equation}
with $\displaystyle{\eta_h(x):=\int_\cM K\PAR{\frac{\NRM{x-y}}{h}}\dd y} $. We will show later (in \Cref{lem:properties-Kh}) that $\eta_h>0$  for $h$ small enough, ensuring that the kernel $K_h$ is well-defined. We consider the following estimator of the density $p$ of $\mu$ obtained by convolution of the occupation measure $\mu_T$ with $K_h$. For $x\in \cM$, 
\begin{align}\label{eq:def-p_{T,h}}
	p_{T,h}(x)&:=  
	\int_\cM K_h(z,x)\mu_T(\dd z)=\frac{1}{T}\int_0^TK_h\PAR{X_s,x}\dd s \nonumber\\
 &=\frac{1}{T}\int_0^T\frac{1}{\eta_h(X_s)}K\PAR{\frac{\NRM{X_s-x}}{h}}\dd s.
\end{align}
Thanks to the definition of $\eta_h$, we notice that
\[
\int_\cM p_{T,h}(x)\dd x=1.
\]
However, the function $p_{T,h}$ is not necessarily a density, as it may not be nonnegative everywhere (recall that $K$ is a signed function). Still, we will show  that the function $p_{T,h}$ approximates the density $p$, and is therefore a density  with high probability. 
Let $x_0$ be an arbitrary fixed point of $\cM$. We  introduce two random measures $\mu_{T,h}$, and $\widehat{\mu}_{T,h}$ on $\cM$, defined by
\begin{align}\label{def:mu_Th}
    \mu_{T,h}(\dd x) = p_{T,h}(x)  \dd x,\text{ and }
    \widehat{\mu}_{T,h}=\begin{cases}
	\mu_{T,h}    &\text{ if }\mu_{T,h} \text{ is a nonnegative measure,}\\
 \delta_{x_0} &\text{ otherwise.}
\end{cases}
\end{align}
The measure $\widehat{\mu}_{T,h}$  is introduced for purely technical purposes. Indeed, the measure $\mu_{T,h}$ may not be a probability measure (with exponentially small probability), so that the risk $\cW_2(\mu_{T,h},\mu)$ may not even be defined, whereas $\cW_2(\widehat{\mu}_{T,h},\mu)$ is always defined.

\begin{remark}\label{rk:K-geodesic}
    It would have been arguably more convenient to work with a kernel based on the geodesic distance $\rho$, i.e. a kernel $\widetilde{K}_h$ defined by $\widetilde{K}_h(x,y):=\frac{1}{\tilde{\eta}_h(x)}K\PAR{\frac{\rho(x,y)}{h}}$, 
with $\widetilde{\eta}_h(x)=\int_\cM K\PAR{\frac{\rho(x,y)}{h}}\dd y $. For instance,  we can easily prove  that $\PAR{h^{-d}\widetilde{\eta}_h}_{h>0}$ converges to $1$ uniformly on $\cM$, with a speed of convergence of order $h^2$. Such a property also holds for $\eta_h$, but only for sufficiently smooth  kernels $K$ satisfying some moments assumptions (see \Cref{def:kernel}). 
 However, for statistical purposes, the use of the Euclidean distance in \eqref{eq:def-K_h} seems natural in the context where the manifold $\cM$ (and hence the geodesic distance $\rho$) is unknown. Yet, the study of the convolution with the kernel $\widetilde{K}_h$ 
is of own interest. Moreover,  remark that since  the manifold $\cM$ is assumed to be compact, the geodesic distance $\rho(\cdot,\cdot)$ on $\cM$ and the Euclidean distance $\| \cdot \|$ of $\R^m$ are known to be equivalent, see e.g. \cite[Proposition 2]{Trillos2020}.
\end{remark}

We are now in position to state our main result, which gives the rate of convergence of $\widehat{\mu}_{T,h}$ to $\mu$ when the diffusion path $(X_t)_{t\in[0,T]}$  exploring the manifold has generator $\cA$. More precisely, our purpose is to upper-bound
\begin{equation}\label{eq:uniform_risk}
    \sup_{x \in \cM } \mathbb{E}_x\SBRA{\W_2^2(\widehat{\mu}_{T,h},\mu)}.
\end{equation} 
For any initial measure $\mu_0$, we can write $\E_{\mu_0}$ as a mixture $\E_{\mu_0}[\cdot ]  = \int \E_x [ \cdot ]  \mu_0(\dd x)$. Hence, a bound on the uniform risk defined in \eqref{eq:uniform_risk} automatically implies a bound on the risk for \textit{any} initial measure. 
As often, such a bound is obtained by decomposing the loss into a bias term $\W_2^2(\mu_h,\mu)$  and a variance term $\mathbb{E}_x\SBRA{\W_2^2(\widehat{\mu}_{T,h},\mu_h)}$, where 
\begin{align}
	\label{eq:def_convol} \mu_h(\dd x)=p_h(x)\dd x \quad \text{ with}\quad p_h(x)&:=
	\int_\cM K_h\PAR{z,x}p(z)\dd z\\
 &=\int_\cM\frac{1}{\eta_h(z)}K\PAR{\frac{\NRM{z-x}}{h}}p(z)\dd z .\nonumber
\end{align}
is the intensity measure of the random measure $\mu_{T,h}$ when the distribution of $X_0$ is the invariant measure $\mu$.

The control of the variance term $\mathbb{E}_x\SBRA{\W_2^2(\widehat{\mu}_{T,h},\mu_h)}$  relies on fine spectral properties of the generator of the diffusion. The proof of the following result is detailed in Section~\ref{sec:proof-mainA} for a diffusion starting from its invariant measure $\mu$ and in Section~\ref{sec:general-distribution} for a diffusion starting from a general initial distribution. The bound on the variance in the following theorem depends on the ultracontractivity constant $u_{\cA}$ of the generator $\cA$, defined in \Cref{sec:general-distribution}.

\begin{theorem}[Estimation from a diffusion with generator $\cA$]\
\label{thm:mainA}
Let $d\ge 1$ and $p$ be a positive $\cC^2$ density function with associated measure $\mu$.  
Let $(X_t)_{t\ge 0}$ be a diffusion with generator $\cA$ satisfying Assumption~\ref{assump:A}. Let $T\ge 2$ and let $0<h\le h_0$ for some constant $h_0$ depending on $\cM$ and $K$. Assume that either $K$ is nonnegative or that   $d\ge 4$ and that $Th^{d}\ge c\ln(T)$ (in which case, $h_0$ additionally depends on  $p_{\min}$ and on the $\cC^1$-norm of $p$). Then,
\begin{equation}
    \sup_{x\in \cM} \mathbb{E}_x\SBRA{\W_2^2(\widehat{\mu}_{T,h},\mu_h)} \le c_0 \frac{u_\cA p_{\max}^2}{p_{\min}^2} \|K\|_\infty^2 \begin{cases}
\frac{h^{4-d}}{T} &\text{if }d\ge 5\\
\frac{\ln(1/h)}T &\text{if } d=4\\
\frac{1}T &\text{if }d\le 3,
    \end{cases}
\end{equation}
where 
$c_0$ depends on $\cM$, and $c$ depends on $\cM$, $K$, $p_{\min}$, $p_{\max}$ and $\kappa_{\min}$. 
\end{theorem}

The second term in the risk decomposition is the bias term $\W_2^2(\mu_h,\mu)$, which was already studied by Divol in \cite{Divol2022}. 
Let $\ell\ge 0$. We introduce the Sobolev space $H^\ell(\cM)$ as the completion of the set of smooth functions on $\cM$ with respect to the norm: 
\[\|f\|_{H^\ell(\cM)}:= \max_{0\le i\le \ell} \Big(\int_{\cM} \|\nabla^if(x)\|^2\ \dd x\Big)^{1/2}.\]
As $\cM$ is compact, we note that for any  $f\in\cC^\ell(\cM)$, $\|f\|_{H^\ell(\cM)}\le  \|f\|_{\cC^\ell(\cM)}$ and $\cC^\ell(\cM)$ is a subset of $H^\ell(\cM)$.

Under additional technical conditions on the kernel $K$ (recalled in \Cref{A:bias}), Divol showed that if $p\in H^\ell(\cM)$ for some $\ell\ge 0$, then
\begin{equation}\label{eq:bound_bias}
    \W_2^2(\mu_h,\mu) \le c_1\frac{\|p\|_{H^\ell(\cM)}^2}{p_{\min}^2} h^{2\ell+2},
\end{equation}
where $c_1$ depends only on $\cM$ and $K$, see  \Cref{prop:bias}. 
As a corollary of Theorem~\ref{thm:mainA} and \eqref{eq:bound_bias}, we obtain a tight control on the risk of the estimator $\widehat{\mu}_{T,h}$.

\begin{corollary}\label{cor:rate_estimator}
Let $d\ge 5$ and $p$ be a positive $\cC^2$ density function with associated measure $\mu$. Further assume that $p$ has a controlled Sobolev norm $\|p\|_{H^\ell(\cM)}$ for some $\ell\ge 2$.   Assume that $K$ is a kernel of order larger than $\ell$ (in the sense of \Cref{def:kernel}). 
Let $(X_t)_{t\ge 0}$ be a diffusion with generator $\cA$ satisfying Assumption~\ref{assump:A}. Let $T\ge 2$ and let $0<h\le h_0$  and assume that $Th^{d}\ge c\ln(T)$, where $h_0$, $c$ are the constants from \Cref{thm:mainA}. Then,
\begin{equation}
   \sup_{x\in \cM} \mathbb{E}_x\SBRA{\W_2^2(\widehat{\mu}_{T,h},\mu)} \le 2 \ c_1\frac{\|p\|_{H^\ell(\cM)}^2}{p_{\min}^2} h^{2\ell+2} +  2\ c_0\frac{ u_\cA p_{\max}^2}{p_{\min}^2} \|K\|_\infty^2 \frac{h^{4-d}}{T}
\end{equation}
where $c_0$ is the constant from \Cref{thm:mainA} and $c_1$ is the constant in \eqref{eq:bound_bias}. 
In particular, for $h$ of order $T^{-1/(2\ell+d-2)}$, it holds that 
\begin{equation}
    \sup_{x\in \cM} \mathbb{E}_x\SBRA{\W_2^2(\widehat{\mu}_{T,h},\mu)}\lesssim T^{-\frac{2\ell+2}{2\ell+d-2}}.
\end{equation}
\end{corollary}
The results of Theorem \ref{thm:mainA} and Corollary \ref{cor:rate_estimator} are non-asymptotic results. Notice that in Corollary \ref{cor:rate_estimator}, the choice of $h$ of order $T^{-1/(2\ell+d-2)}$ is compatible with the condition $Th^{d}\ge c\ln(T)$ required in Theorem \ref{thm:mainA}. It is a remarkable fact that the constants in the previous corollary only depend on the generator $\cA$ through the uniform ellipticity constant $\kappa_{\min}$ and the ultracontractivity constant $u_\cA$. From a statistical perspective, this implies that the knowledge of the exact SDE satisfied by the sample path $(X_t)_{t\in [0,T]}$ is not needed to estimate the invariant measure $\mu$. Only an a priori estimate on the uniform ellipticity constant $\kappa_{\min}$  of the generator $\cA$ of the sample path and its ultracontractivity constant $u_\cA$ have to be known. For instance, the same reconstruction method will apply for estimating a sample path with either of the generators $\cL$ or $\cA_{pq}$ for $q=p/2$ mentioned in the introduction.

Comparing with the results of Wang and Zhu in \cite{Wang2023} for the operator $\cL$, we note that for $d\ge 5$, the rate $T^{-\frac{2\ell+2}{2\ell+d-2}}$ is faster than the rate of $T^{-\frac{2}{d-2}}$ that they obtained for the occupation measure $\mu_T$, see \eqref{eq:rate_WZ}. Actually, our results allow us to recover their rate of convergence, for \textit{any} generator $\cA$ satisfying \Cref{assump:A}.

\begin{corollary}\label{cor:empirical_measure}
    Let $d\ge 1$ and $p$ be a positive density function of class $\cC^2$. 
Let $(X_t)_{t\ge 0}$ be a diffusion with generator $\cA$ satisfying Assumption~\ref{assump:A}. Then, for all $T\ge 2$,
\begin{equation}
   \sup_{x\in \cM} \mathbb{E}_x\SBRA{\W_2^2(\mu_T,\mu)} \le  c_0\left(1+\frac{u_{\cA}p^2_{\max}}{p^2_{\min}}\right)\begin{cases}
 T^{-1} & \mbox{ when }d\le 3 \\
 T^{-1} \ln(1+T) & \mbox{ when }d=4\\
 T^{-\frac{2}{d-2}} & \mbox{ when }d\ge 5,
\end{cases}  
\end{equation}
for some constant $c_0$ depending on $\cM$.
\end{corollary}
\begin{proof}
    Let $K$ be a nonnegative Lipschitz-continuous kernel supported in $[0,1]$ with $\int_{\R^d}K(\|u\|)\dd u=1$. As $K$ is positive, $\mu_{T,h}$ is always a probability measure, so that $\widehat{\mu}_{T,h}=\mu_{T,h}$. For any probability measure $\nu$, we define its convolution $\nu_h$ by $K_h$ as in \eqref{eq:def_convol}. We prove in \Cref{lem:convolution_wass} that for any measure $\nu$ and $h>0$ small enough,  {$\W_2^2(\nu_h,\nu)\le  c_0 h^2$ for some constant $c_0$ depending on $\cM$}. Hence, both $\W_2^2(\mu_T,\mu_{T,h})$ and $\W_2^2(\mu,\mu_{h})$ are of order $h^2$.  Then, 
    \begin{align*}
        \mathbb{E}_x\SBRA{\W_2^2(\mu_T,\mu)} &\le 4\mathbb{E}_x\SBRA{\W_2^2(\mu_T,\mu_{T,h})} +4\mathbb{E}_x\SBRA{\W_2^2(\mu_{T,h},\mu_h)}+ 4\W_2^2(\mu_h,\mu) \\
        &\le 8 c_0h^2  +4\mathbb{E}_x\SBRA{\W_2^2(\mu_{T,h},\mu_h)}.
    \end{align*} 
    We pick $h=T^{-1/(d-2)}$ for $d\ge 5$ and $h=T^{-1/2}$ for $d\le 4$ and apply \Cref{thm:mainA} to conclude.
\end{proof}

Finally, we address the optimality of our statistical procedure using  minimax theory. Consider a class $\cP_T$ of probability distributions of diffusion processes $(X_t)_{t\in [0,T]}$. For $\P_T\in \cP_T$, we let $\mu(\P_T)$ be the invariant measure of the diffusion process $(X_t)_{t\ge 0}$ whose restriction to $[0,T]$ has distribution $\P_T$. The associated minimax rate of convergence is defined as 
\begin{equation}\label{eq:minimax}
  \cR(\cP_T)=  \inf_{\hat \mu}\sup_{\P_T\in \cP_T} \E_{\P_{T}}[\W_2(\hat\mu, \mu(\P_T))],
\end{equation}
where the infimum is taken over all estimators $\hat \mu$, i.e. measurable functions of the observation $(X_t)_{t\in [0,T]}$, and where the expectation is over processes $(X_t)_{t\in [0,T]}$ distributed as $\P_T$. 
To put it another way, the minimax rate is the best risk an estimator can attain (uniformly on $\mathcal{P}_T$) for the problem of estimating the invariant measure of some diffusion process $(X_t)_{t\in [0,T]}$ whose distribution lies in $\cP_T$.

For parameters $\ell\ge 2$, $\kappa_{\min}$, $p_{\min}$, $p_{\max}$, $u_{\max}$, $L>0$, 
we consider the statistical model $\cP_{T,\ell}=\cP_{T,\ell}(\kappa_{\min},p_{\min},p_{\max},u_{\max}, L)$ consisting of all the distribution of diffusion processes $(X_t)_{t\in [0,T]}$ on $\cM$ with arbitrary initial distribution, generator $\cA$ satisfying \Cref{assump:A} with constant $\kappa_{\min}$ and ultracontractivity constant $u_\cA$ smaller than $u_{\max}$, and whose invariant measure $\mu$  has a $\cC^2$ density $p$ on $\cM$ with $\|p\|_{H^\ell(\cM)}\le L$ and $\|p\|_{\cC^1(\cM)}\le L$,  satisfying $p_{\min}\le p\le p_{\max}$. 
Remark that the kernel density estimator $\widehat{\mu}_{T,h}$ attains (for $d\ge 5$) the rate of convergence $T^{-\frac{\ell+1}{2\ell+d-2}}$ uniformly on the class $\cP_{T,\ell}$. Hence, the minimax rate satisfies $\cR(\cP_{T,\ell})\lesssim T^{-\frac{\ell+1}{2\ell+d-2}}$. The next proposition, proved in \Cref{sec:minimax}, states that this rate cannot be improved.

\begin{proposition}\label{prop:minimax}
    Let $\ell \ge 2$ be an integer. Then, for $\kappa_{\min}$, $p_{\min}$ small enough and $p_{\max}$, $u_{\max}$, $L$ large enough, {and all $T\ge 1$,}
    \begin{equation}
        \cR(\cP_{T,\ell}) \gtrsim \begin{cases}
                     T^{-1/2}& \text{ if }d\le 4\\
             T^{-\frac{\ell +1}{2\ell+d-2}}& \text{ if }d\ge 5.
        \end{cases}
    \end{equation}
\end{proposition}
For $d\ge 5$, these rates match with the rates obtained by our kernel-based estimator $\widehat{\mu}_{T,h}$, whereas for $d\le 4$, the empirical estimator $\mu_T$ attains the minimax rate (up to logarithmic factors), according to the results of Wang and Zhu \cite{Wang2023} for the Langevin generator $\cL$, or according to  \Cref{cor:empirical_measure} for a general generator $\cA$.

\begin{remark}
    The estimator $\widehat{\mu}_{T,h}$ has a density with respect to the volume measure $\dd x$ on $\cM$. As a consequence,  computing $\widehat{\mu}_{T,h}$  requires the knowledge of the manifold $\cM$, prohibiting the use of this method in the situation where the manifold $\cM$ is unknown.  However, we expect that similar methods to  the ones developed in \cite{Divol2022} will allow us to create an estimator $\widehat{\mathrm{vol}}_{\cM}$ of the volume measure. Such an estimator of the volume measure is based on a patch-based reconstruction of the underlying manifold ${\cM}$ developed by Aamari and Levrard \cite{aamarilevrard2019}. Guarantees for this method are only known in the case of i.i.d. samples, although we believe that the results of \cite{aamarilevrard2019} could be adapted to the setting of this paper, where a diffusion path is observed. Such an estimator of the volume measure could then be used through a plug-in method to design a new estimator $\widetilde \mu_{T,h}$ that would not require the knowledge of the manifold $\cM$. 
\end{remark}

\begin{remark}
{
An arguably more natural variant of the Wasserstein estimation problem is the estimation of the invariant measure $\mu$ with respect to the $L^2$-metric. Nonetheless, since we consider the most relevant end goal to be the setting where the manifold $\cM$ is unknown, the Wasserstein metric appears to be more suitable: indeed, unlike the $L^2$-metric which is not well-defined for measures with disjoint supports, the Wasserstein metric can be small for two measures supported on  nearby manifolds. }

{
However, we believe that estimation with respect to the $L^2$-metric is still feasible and, in fact, the analysis is simplified compared to the $\cW_2$-metric for several reasons. First, the $L^2$-metric is well-defined even when $\mu_{T,h}$ is not a probability measure, so there is no need to control the probability that $\mu_{T,h}$ is nonnegative (\Cref{sec:Proba}). Second, the proof of \Cref{prop:variance} becomes more straightforward, as \Cref{lemma: var-lem2} and subsequently \Cref{lem:wang} can be directly applied to the function $\mu_{T,h} - \mu$, avoiding the more intricate use of \Cref{prop:1/2-op}. This approach leads to the conclusion that the variance of $\mu_{T,h}$ is of order $h^{2-d}/T$ (for $d \ge 3$). Standard Taylor expansions (e.g., using tools such as \cite[Lemma 10]{Divol2022}) can also be employed to show that the bias of the estimator is of order $h^\ell$ when the density $p$ belongs to $H^\ell$. Altogether, this yields a decomposition of the form
\begin{equation}
    \E[\|\mu_{T,h} - \mu\|_{L^2}] \lesssim h^\ell + \frac{h^{1 - d/2}}{\sqrt{T}}.
\end{equation}
For $h$ of the order of $T^{-1/(2\ell + d - 2)}$, this leads to an estimation rate of order $T^{-\ell/(2\ell + d - 2)}$. This matches the $L^2$ estimation rates obtained in the Euclidean setting, as studied by \cite{dalalyan2006asymptotic, dalalyan2007asymptotic} and \cite{amorinogloter}, among others. We leave the details to the interested reader.
}
\end{remark}

\section{Preliminaries}\label{sec:preliminaries}

We detail in this section some tools, notation, and general results on elliptic operators on compact manifolds. Recall that $\cM$ is a smooth compact $d$-dimensional connected Riemannian manifold without boundary.

\subsection{ The Laplace-Beltrami operator}\label{sec:weighted-laplacian}

By Green's theorem, the Laplace-Beltrami operator is symmetric, with for all $f,g\in \cC^2(\cM)$,
\begin{equation}\label{eq:green-thm}
 \int_{\cM} (\Delta f)g \dif x =-\int_{\cM} \SCA{ \nabla f , \nabla g }  \, \dd x=  \int_{\cM} f(\Delta g) \dif x. 
 \end{equation}
The  operator $\Delta$ defines an essentially self-adjoint operator over $L^2(\dd x)$ with a discrete spectrum (see e.g.  \cite[Chapter I, Section 3]{Chavel1984}).  
We denote by
 \[
 0 = \lambda_0 < \lambda_1 \le \lambda_2 \le \lambda_3 \le \cdots
 \]
 the eigenvalues of $(-\Delta)$, and by $(\phi_i)_{i\ge 0}$ the respective eigenfunctions  of $\Delta$ (note that they are of class $\cC^{\infty}$ on $\cM$). The first eigenfunction $\phi_0$ is constant and the family $(\phi_i)_{i\ge 0}$ forms a Hilbert basis of $L^2(\dd x)$: for any $f\in L^2(\dd x)$,
 \[
 f=\sum_{i=0}^{+\infty}\beta_i \phi_i,\qquad \mbox{ with }\qquad \beta_i=\int_\cM f\phi_i\dd x.
 \]
The first nonzero eigenvalue $\lambda_1$ is of particular importance and is called the spectral gap.

 The inverse operator $\Delta^{-1}$ is defined on $L^2_0(\dd x):=\BRA{f\in L^2(\dd x): \int_\cM f\dd x=0}$ by 
 \begin{equation}
 \Delta^{-1}(f):=-\sum_{i=1}^\infty\frac{\beta_i}{\lambda_i}\phi_i.
\end{equation}
We introduce the operator $\PAR{-\Delta}^{-1/2}$ 
defined   for $f\in L_0^2(\dd x)$ by
\[
\PAR{-\Delta}^{-1/2}(f):=\sum_{i=1}^\infty\frac{\beta_i}{\sqrt{\lambda_i}}\phi_i.
\]

\subsection{Green function of the Laplace Beltrami operator}\label{sec:green}

The Green function $G$ of $\Delta$ is a linear bounded operator $G:L^2(\dd x)\to L^2(\dd x)$ which is, in some sense, an inverse of $\Delta$ in $L^2(\dd x)$.

\begin{proposition}[See Appendix A in \cite{Alesker2013}, or Theorem 4.13 in \cite{Aubin1982}]
	\label{prop:greenfunction}
	We define $\mathrm{diag}(\cM)=\BRA{(x,x):x\in\cM}$.
	There exists a unique continuous function  $G \in\cC^{\infty}\PAR{\cM\times\cM \setminus\mathrm{diag}(\cM)}$, which has the following properties
	\begin{enumerate}[label=(\roman*)]
		\item $\forall x\in\cM$, $G(x,\cdot)\in L^1(\dd x)$  with $\displaystyle{\int_\cM G(x,\cdot)\dd x=0}$;
		\item $G$ is symmetric : $\forall (x,y)\in\cM^2\setminus\mathrm{diag}(\cM)$, $G(x,y)=G(y,x)$;
		\item $\forall f\in\cC^2(\cM), \forall x\in \cM$, we have
		\[
		\int_{\cM}G(x,y)\Delta  f(y)\dd y=f(x)-\int_{\cM}f(y)\dd y;
		\]
  \item  there exists a constant $\kappa>0$  such that $\forall(x,y)\in\cM^2\setminus \mathrm{diag}(\cM)$,
\[\ABS{G(x,y)}\le \kappa\begin{cases}
 1& \text{ when $d=1$,}\\
1+\ABS{\ln \rho(x,y)} & \text{ when $d=2$, and}\\
\rho(x,y)^{2-d}&\text{ when $d\ge 3$.}
\end{cases}\]
\end{enumerate}
	The function $G$ is called the Green function of the operator $\Delta$ and it naturally defines an operator from $L^2(\dd x)$ to $L^2(\dd x)$, which is also denoted by $G$, as follows: $\forall f\in L^2(\dd x)$, $\forall x\in \cM$,
	\begin{equation}\label{eq:green-op}
		\PAR{Gf}(x):=\int_{\cM}G(x,y)f(y)\dd y.
	\end{equation}
\end{proposition}

From Proposition~\ref{prop:greenfunction}~$(iii)$ and \eqref{eq:green-op}, we remark that for  $f\in L^2_0(\dd x)$ we have
\[
Gf=\Delta^{-1}(f).
\]

The Green function will play a central role to control the variance of $\widehat{\mu}_{T,h}$, in particular through the following lemma, which controls the behavior of $Gf$ for a function $f$ localized on a small ball.

\begin{lemma}\label{lem:estimation-green}
	Let $h>0$ sufficiently small. For any $(x,z)\in\cM^2$, any continuous function $f$ with support in the ball $\cB(x,h)$, there exist constants $\kappa_1,\kappa_2>0$ depending only on $\cM$  such that:
 \begin{itemize}
     \item when $d=1$, $\ABS{(Gf)(z)}\le \kappa_1 \NRM{f}_{\infty} h$,
     \item when $d=2$, $\ABS{(Gf)(z)}\le \kappa_1 \NRM{f}_{\infty} h^2(1+\ln |h|)$,
     \item  when $d\ge 3$,
	\begin{align*}
		\ABS{(Gf)(z)}&\le 
		\kappa_1 \NRM{f}_{\infty}h^2\quad \text{ when }\rho(x,z)\le 2h,\\
		\ABS{(Gf)(z)}&\le \kappa_2\NRM{f}_\infty h^d \rho(x,z)^{2-d}\quad \text{ when }\rho(x,z)>2h.
	\end{align*}
 \end{itemize}
\end{lemma}

The proof, given in Appendix~\ref{A:estimation-green-proof}, uses the  Riemannian normal parametrization of the manifold $\cM$.

\subsection{The  elliptic operator $ \mathcal{A}$ }\label{sec:elliptic}

Consider an elliptic operator $\cA$ on $\cM$ satisfying Assumption~\ref{assump:A}. As mentioned previously, using the \textit{carré du champ} $\Gamma$ defined by \eqref{eq:carre-du-champ}, $\cA$ satisfies a Green's formula~\eqref{eq:green-formula-generalA}. 
    Furthermore, as a uniformly elliptic operator  of second order on a compact manifold without boundary,  symmetric with respect to the measure $\mu$, the operator $\cA$ is  essentially self-adjoint with respect to $L^2(\mu)$, and its spectrum is discrete (see {Appendix \ref{appendix:selfadjoint}, also} e.g. \cite[Chapter 6]{Evans1998} and \cite{toan-thesis}). Consequently, there exist functions $\PAR{\psi_i }_{ i \in \mathbb{N}} $ and a sequence of real numbers $0=\gamma_0 < \gamma_1 \le \gamma_2 \le \cdots $  such that
\begin{enumerate}[label=(\roman*)]
 	\item[(i)] for each $i$, $\psi_i$ is an eigenfunction of $-\cA$ associated to the eigenvalue $\gamma_i$ (counted with multiplicity);
 	\item[(ii)] $\PAR{ \psi_i}_{i\in \mathbb{N}}$ is an orthonormal basis of $L^2(\mu)$.
 \end{enumerate}

Let $L^2_0(\mu):=\BRA{f\in L^2(\mu): \int_\cM f\dd \mu=0}$. As for the Laplace-Beltrami operator, 
we introduce the operators  $\cA^{-1}$ and $\PAR{-\cA}^{-1/2}$ defined for $f\in L^2_0(\mu)$ by
 \begin{equation}\label{eq:A^1-A^-1/2}
 \cA^{-1}(f):=-\sum_{i=1}^\infty\frac{\alpha_i}{\gamma_i}\psi_i \quad \text{and}\quad 
 \PAR{-\cA}^{-1/2}(f):=\sum_{i=1}^\infty\frac{\alpha_i}{\sqrt{\gamma_i}}\psi_i,
 \end{equation}
  with 
 $\alpha_i=\int_\cM f\psi_i\dd \mu$.
At last, we denote by $(P_t)_{t\ge 0}$ the semigroup of a diffusion process $(X_t)_{t\ge 0}$ with generator $\cA$.  
  Then, for $f\in L^2(\mu)$, $f=\sum_{i\ge 0}\alpha_i\psi_i$ with $\alpha_i=\int_\cM f\psi_i\dd \mu$, we have
\[
\forall x \in \cM,\ P_tf(x)=\E_x\SBRA{f(X_t)}=\sum_{i=0}^\infty\mathrm{e}^{-\gamma_i t}\alpha_i\psi_i(x).
\]  
This result is a consequence of the
 Dynkin formula \cite{Kulik2018}, which implies that for any function $f$ in the domain of $\cA$ 
\[
\forall x \in \cM,\ P_tf(x)=\E_x[f(X_t)]=f(x)+\int_0^tP_s\cA f(x)\dd s.
\]

We now give a Poincaré inequality for the general operator $\cA$.
\begin{proposition}(Poincaré's inequality for $\cA$)\label{prop:poincare-A}
For any $f \in \mathcal{C}^2(\cM)$  such that $\int_\cM f\dd \mu=0$, 
$$   \int_\cM -f \mathcal{A}f \mathrm{d}\mu \ge \frac{p_{\min}\kappa_{\min}}{ p_{\max}}\lambda_1 \int_\cM f^2 \mathrm{d}\mu,$$
where $\lambda_1$ is the spectral gap of $\Delta$, and $\kappa_{\min}$ is defined in \eqref{eq:A-def-kappa}.
\end{proposition}
\begin{proof}
    By the symmetry of $\mathcal{A}$, Equation~\eqref{eq:A-def-kappa} and the Poincaré's inequality of $\Delta$, we have:

    \begin{align*}
    \int_\cM -f \mathcal{A}f \mathrm{d}\mu&= \int_\cM \Gamma(f,f) \mathrm{d} \mu \ge p_{\min}\kappa_{\min}\int_\cM |\nabla f|^2 \mathrm{d} x \ge p_{\min}\kappa_{\min}  \lambda_1 \int_\cM (f-\overline{f})^2 \mathrm{d} x
        \\
 &\ge  \frac{p_{\min}\kappa_{\min}}{p_{\max}}\lambda_1 \int_\cM (f-\overline{f})^2 \mathrm{d}\mu \ge \frac{p_{\min}\kappa_{\min}}{p_{\max}}\lambda_1 \int_\cM f^2 \mathrm{d}\mu, 
    \end{align*}
where $\overline{f}:= \int_\cM f \mathrm{d} x.$ We  develop the square and use that that $\int_\cM f\dd \mu=0$ to obtain the conclusion.
\end{proof}

\begin{remark}\label{rk:Apq}
The elliptic operators  $\cA_{pq}$ and $\cL$ defined in \Cref{sec:intro} are essentially Laplace operators.
Indeed, the operator $\cL$  is a weighted Laplacian with weight $\mu(\dd x)=p(x)\dd x$ as defined  in \cite[Section 3.6]{Grigoryan2009}, while the operator $\cA_{pq}$ is a weighted Laplacian on a weighted manifold, i.e. a manifold $\cM$ with a new Riemannian metric depending on the weight $q$, see Section 3.6 and Exercise 3.11 in \cite{Grigoryan2009}.
\end{remark}

\section{Variance term for the stationary process}\label{sec:proof-mainA}

Before proving \Cref{thm:mainA}, we state a simpler version of the theorem, where the initial measure $\delta_{x}$ for the sample path $(X_t)$ is replaced by the invariant measure $\mu$. For such a choice, it holds that for all $y\in \cM$,
\begin{equation}\label{eq:bias}
\E_\mu\SBRA{p_{T,h}(y)}=\int_\cM K_h(z,y)p(z)\dd z=p_h(y).
\end{equation}
 We will then explain in Section~\ref{sec:general-distribution} how we can extend the result to any initial distribution $\mu_0$  using the ultracontractivity of the semi-group $(P_t)_{t\ge 0}$.

\begin{proposition}\label{prop:variance}
Let $d\ge 1$ and $p$ be a positive $\cC^2$ density function with associated measure $\mu$.  
Let $(X_t)_{t\ge 0}$ be a diffusion with generator $\cA$ satisfying Assumption~\ref{assump:A}. Let $0<h\le h_0$ for some constant $h_0$ depending on $\cM$ and $K$. Assume that either $K$ is nonnegative or that   $d\ge 4$ and that $Th^{d-2}\ge c\ln(T)$ (in which case, $h_0$ additionally depends on  $p_{\min}$ and on  the $\cC^1$-norm of $p$). Then,
\begin{equation}
    \mathbb{E}_\mu\SBRA{\W_2^2(\widehat{\mu}_{T,h},\mu_h)} \le c_0\, \frac{p_{\max}^2}{p_{\min}^2} \|K\|_\infty^2 \begin{cases}
\frac{h^{4-d}}{T} &\text{if }d\ge 5\\
\frac{\ln(1/h)}T &\text{if } d=4\\
\frac{1}T &\text{if }d\le 3,
    \end{cases}
\end{equation}
where 
$c_0$ depends on $\cM$, and $c$ depends on $\cM$, $K$, $p_{\min}$, $p_{\max}$ and $\kappa_{\min}$.
\end{proposition}
The main difference between \Cref{prop:variance} and \Cref{thm:mainA} is the choice of the initial measure, which is of the form $\delta_x$ for \Cref{thm:mainA} and is equal to $\mu$ in \Cref{prop:variance}. As for any distribution $\mu_0$ and random variable $U$, $\E_{\mu_0}[U]=\int \E_x[U]\dd \mu_0(x)$, \Cref{thm:mainA} is stronger, and implies that the convergence holds for \textit{any} initial measure $\mu_0$.

The remainder of this section is dedicated to proving \Cref{prop:variance}. 
We first state some useful properties on $K_h$. The following lemma is stated in \cite[Lemma 10]{Divol2022} under stronger regularity hypotheses on the kernel $K$. We can relax these assumptions and a proof is given in Appendix \ref{A:proof-properties-Kh}. The first point of the lemma  guarantees  that $K_h$ is well defined on $\cM^2$ for $h$ small enough.

\begin{lemma}\label{lem:properties-Kh}
Assume that $K$ is a continuous function on $\dR$ with support in $[0,1]$, such that $\int_{\dR^d}K(\NRM{z})\dd z=1$. Let $h>0$, and consider $K_h$ defined by \eqref{eq:def-K_h} with the renormalizing factor $\eta_h$. Then,
\begin{enumerate}[label=(\roman*)]
\item $h^{-d}\eta_h$ converges to $1$ uniformly on $\cM$;
\item {there exists $h_c>0$, depending on $\cM$ and $K$, such that $\forall h<h_c$, $K_h$ is well defined and} $\forall x\in \cM, \displaystyle{\int_\cM K_h(x,y)\dd y=1}$; 
\item  $\forall h<h_c$, $K_h$ is bounded on $\cM\times \cM$ with  $\NRM{K_h}_\infty\le 2\NRM{K}_\infty h^{-d}$;
\item if furthermore $K$ is Lipschitz continuous, then $\forall h<h_c$, and all $x\in \cM$,  $y\mapsto K_h(x,y)$ is Lipschitz continuous with constant $2\mathrm{Lip}(K)h^{-d-1}$, where $\mathrm{Lip}(K)$ is the Lipschitz constant of $K$.
\end{enumerate}
\end{lemma}

We also require the following elementary convergence result,
proved in Appendix~\ref{A:proof-ph}.

\begin{lemma}\label{lem:ph-to-p}
Let $p_h$ be defined by \eqref{eq:def_convol} for $h>0$. Under the assumptions of Lemma \ref{lem:properties-Kh}, when $h$ goes to $0$, $\PAR{p_h}_{h>0}$ converges to $p$ uniformly  on $\cM$. Moreover, there exists $h_c$ depending on $\cM$, $K$, $p_{\min}$ and the $\cC^1$-norm of $p$ such that for all $0<h\le h_c$, 
$\inf_{y\in\cM}p_h(y)\ge \frac{p_{\min}}{2}$ and $\sup_{y\in \cM} p_h(y)\le 2p_{\max}$.
\end{lemma}

Hence, for $h$ small enough, $\mu_h$ is indeed a probability measure. Recall that 
the function $K$ is a signed kernel, so that $\mu_{T,h}$ is a priori a signed measure. We introduce the event $E_{T,h}$ defined by
\begin{equation}\label{eq:event}
	E_{T,h}=\BRA{ p_{T,h}\ge 0 }.
\end{equation}
On this event, we have $\widehat{\mu}_{T,h}=\mu_{T,h}$ and we notice that for $0<h<h_c$,
\begin{equation}\label{eq:E+E^c}
\E_\mu\SBRA{\W_2^2\PAR{\widehat{\mu}_{T,h},\mu_h}}\le \E_\mu\SBRA{\W_2^2\PAR{\mu_{T,h},\mu_h}\ind_{E_{T,h}}}+\mathrm{diam(\cM)}^2\P_\mu(E_{T,h}^c),
\end{equation}
where $\mathrm{diam}(\cM)$ is the diameter of $\cM$. Of course, when $K$ is nonnegative, $p_{T,h}$ is also nonnegative, so that in that case, the event $E_{T,h}$ is  satisfied for all $h>0$.

To prove \Cref{prop:variance}, we will need several intermediate results  related to the spectral decompositions of the operator $\Delta$, $\cA$, and their inverses. 
We first give a lemma that is a direct consequence from a result by Peyre \cite[Corollary 2.3]{peyre2018comparison} (see also \cite[Section 5.5.2]{santambrogio2015} on the negative Sobolev norm), which links the Wasserstein distance to the inverse Laplace operator. {A new proof of Peyre's result and of the lemma is given in Appendix~\ref{appendix:peyre estimate}.} Let us remind that the inverse operator $\Delta^{-1}$ is defined on $L^2_0(\dd x)=\BRA{f\in L^2(\dd x): \int_\cM f\dd x=0}$ (see Section~\ref{sec:green}).

\begin{lemma}\label{lem:wang}
	Let $f_1,f_2\in L^2(\dd x)$ be two probability density functions with respect to the volume measure $\dd x$, with $f_1$ lower bounded by some positive constant $f_{\min}>0$. Then, we have
	\begin{equation}
		\W_2^2(f_1 \dd x,f_2 \dd x)\le \frac{4}{f_{\min}}\int_{\cM}\ABS{ \PAR{-\Delta}^{-1/2}\PAR{f_1-f_2}}^2\dd x.
	\end{equation}
\end{lemma}

Besides, ellipticity yields the following relation between the general operator $\cA$ and the Laplace-Beltrami operator $\Delta$.

\begin{lemma}\label{lemma: var-lem 3}
	For any function $f \in L^2_0(\mu)$, we have
	$$\int_{\cM} \ABS{(-\mathcal{A})^{-1/2}f}  ^2 \dd \mu \le  \frac{1}{p_{\min}\kappa_{\min}} \int_{\cM} \ABS{(-\Delta)^{-1/2}(fp)}  ^2 \dd x.$$
\end{lemma}
\begin{proof}
From Equation \eqref{eq:green-formula-generalA}, $$ \int_{\cM} \Gamma( \mathcal{A}^{-1} f,\mathcal{A}^{-1} f ) \dd \mu  =\int_{\cM}f\PAR{-\cA^{-1}}f\dd \mu= \int_{\cM} \ABS{(-\mathcal{A})^{-1/2}f}  ^2 \dd \mu. $$
Besides, since $\mathcal{A}^{-1} f \in L^2(\mu)$ and since $\cC^1(\cM)$ is dense in this space,
\begin{align*}
	&\sqrt{ \int_{\cM} \Gamma( \mathcal{A}^{-1} f,\mathcal{A}^{-1} f ) \dd \mu } \\
	\le& \sup \Big\{  \int_{\cM} \Gamma(-\mathcal{A}^{-1} f, g ) \mathrm{d}\mu  : g \in \cC^1(\cM) \text{ such that } \int_{\mathcal{M}} \Gamma(g,g) \mathrm{d} \mu \le 1 \Big\}\\
	=&  \sup \Big\{  \int_{\cM} f g \dd \mu : g \in \mathcal{C}^1(\cM) \text{ such that } \int_{\mathcal{M}} \Gamma(g,g) \mathrm{d} \mu \le 1 \Big\}.
\end{align*}
Then, because the simple fact that the supremum of a given set is always bigger than the supremum on any subset, and using \eqref{eq:A-def-kappa}, we have
\begin{align*}
	&\sqrt{ \int_{\cM} \Gamma( \mathcal{A}^{-1} f,\mathcal{A}^{-1} f ) \dd \mu }\\
	\le & \sup \Big\{  \int_{\cM} (fp) g \dd x : g \in \mathcal{C}^1(\cM) \text{ such that } \int_\cM \ABS{\nabla g}^2 p_{\min}\kappa_{\min} \dd x\le 1 \Big\}\\
	=  & \PAR{p_{\min}\kappa_{\min}}^{-1/2} \sup \Big\{  \int_{\cM} (fp) g \dd x : g \in \mathcal{C}^1(\cM) \text{ such that } \int_\cM \ABS{\nabla g}^2  \dd x\le 1 \Big\}. 
\end{align*}
Besides, using Green's theorem, we have  that for all $g \in \cC^1(\cM)$
\begin{align*}
	&  \int_{\cM} (fp) g \dd x = \int_{\cM} g \Delta (\Delta^{-1})(fp)  \dd x =  -\int_\cM \SCA{ \nabla (\Delta^{-1}) (fp) , \nabla g} \dd x.	
\end{align*}
Hence, by Hölder's inequality, we conclude that:
$$ \sqrt{ \int_{\cM} \Gamma( (\mathcal{A}^{-1}) f,(\mathcal{A}^{-1}) f) \dd \mu } \le   \PAR{p_{\min}\kappa_{\min}}^{-1/2} \sqrt{\int_{\cM} \ABS{(-\Delta)^{-1/2}(fp)}  ^2 \dd x }. \qedhere$$
\end{proof}

The following standard result is crucial: it bounds the variance of the random variable $\mu_T(f)$ for some function $f$ in terms of the generator $\cA$.

\begin{lemma}\label{lemma: var-lem2}
	Let $(X_t)_{t\ge 0}$ be a diffusion with generator $\cA$, starting from its invariant measure $\mu$. We have for any $f\in L^2_0(\mu)$,
	\[
	\E_{\mu}\SBRA{\PAR{\frac{1}{T}\int_0^Tf(X_s)\dd s}^2}\le \frac{2}{T}\int_\cM\ABS{\PAR{-\cA}^{-1/2}f}^2\dd \mu.
	\]
		
\end{lemma}


\begin{proof} Recall that $(\gamma_i)_{i\ge0}$ are the eigenvalues of $-\cA$, with $0=\gamma_0<\gamma_1\le\gamma_2\le\cdots$, and $(\psi_i)_{i\ge 0}$ their respective eigenfunctions. We also remind  the reader that $(P_t)_{t\ge 0}$  denotes the semigroup of a process $(X_t)_{t\ge 0}$ with generator $\cA$ (see Section~\ref{sec:elliptic}).

	Since $f\in L^2_0(\mu)$, we write $
	\displaystyle{f=\sum_{i=1}^\infty\alpha_i\psi_i}$ with $\alpha_i=\int_\cM f\psi_i\dd \mu$, and $\alpha_0=0$.
	By the Markov property, denoting by $\PAR{\cF_t}_{t\ge 0}$ the natural filtration of the process $(X_t)_{t\ge 0}$, we have
	\begin{align*}
		\E_\mu\SBRA{\PAR{\frac{1}{T}\int_0^Tf(X_s)\dd s}^2}&=\frac{2}{T^2}\E_\mu\SBRA{\int_0^T\int_s^Tf(X_t)f(X_s)\dd t\dd s}\\
		&=\frac{2}{T^2}\int_0^T\int_s^T\E_\mu\SBRA{\E\SBRA{f(X_t)|\cF_s}f(X_s)}\dd t\dd s\\
		&=\frac{2}{T^2}\int_0^T\int_s^{T}\E_\mu\SBRA{P_{t-s}f(X_s)f(X_s)}\dd t\dd s.
	\end{align*}
	By assumption, the distribution of $X_t$ is $\mu$ for any $ t\ge 0$, and computing the expectation using the link between the semigroup $(P_t)_{t\ge 0}$ and the generator $\cA$ of the process, we then obtain
	\begin{align*}
		\E_\mu\SBRA{\PAR{\frac{1}{T}\int_0^Tf(X_s)\dd s}^2}
		&=\frac{2}{T^2}\sum_{i=1}^\infty\alpha_i^2\int_0^T\int_s^{T} e^{-\gamma_i(t-s)}\dd t\dd s\\
		&\le \frac{2}{T}\sum_{i=1}^\infty\frac{\alpha_i^2}{\gamma_i} .
	\end{align*}
	The lemma is then proved by definition of $(-\cA)^{-1/2}$ given by Equation~\eqref{eq:A^1-A^-1/2}.
\end{proof}

	\subsection{Estimation of the probability $\P_\mu(E_{T,h}^c)$}\label{sec:Proba}

Recall the definition of the event $E_{T,h}=\BRA{ p_{T,h}\ge 0 }$. In view of \eqref{eq:E+E^c}, we need to bound the probability $\P_\mu(E_{T,h}^c)$.
If the kernel $K$ is nonnegative, then $p_{T,h}\ge 0$ and $\P_\mu(E_{T,h}^c)=0$. Otherwise, when $K$ is signed, we will use that
\begin{equation}\label{eq:control-P(E^c)}
\P_\mu(E_{T,h}^c) \le  \P_{\mu}\PAR{\inf_{y\in \cM}p_{T,h}(y)<p_{\min}/8 },
\end{equation}
and provide a bound for the right-hand side of the inequality. 

From Lemma~\ref{lem:ph-to-p}, and in what follows, we 
choose $0<h\le h_c$ so that $ p_{\min}/2\le p_h\le 2p_{\max}$.

	\begin{proposition}\label{prop:bernstein} 
		Assume that $d\ge 4$. There exist $c_0 $ depending on $\cM$, $K$, $p_{\min}$, $p_{\max}$ and $\kappa_{\min}$, and $h_0$ depending on $\cM$, $K$, $p_{\min}$ and the $\cC^1$-norm of $p$ such that  for any $T>0$ and $y \in \cM$, when $0<h\le h_0$, we have
		$$\P_{\mu}\PAR{ p_{T,h}(y) <p_{\min}/4} \le \exp\PAR{- c_0 T h^{d-2} }.$$
	\end{proposition}
	
	\begin{proof}
Let $y\in\cM$ be fixed.	 We first start with a pointwise concentration bound for $p_{T,h}(y)$ around its expectation $p_h(y)$. We apply the Bernstein's bound obtained in \cite[Theorem 3.5]{gao2014bernstein} to the function $g(x):=-K_h\PAR{x,y}+p_h(y)$, with $\Phi(u)=\Psi(u)=u^2/2$ (using the notation of \cite{gao2014bernstein}). Note that $\int_{\cM}g\dd \mu=0$. Write $a_+=\max(a,0)$ for $a\in \R$. Let $M= \|g_+\|_{L^2(\mu)}$ and
\[
 \sigma^2=\lim_{T\to +\infty}\frac{1}{T}\mathrm{Var}_{\mu}\PAR{\int_0^T g(X_s)\dd s}.
 \]
Then, using Poincaré's inequality for $\cA$ given in \Cref{prop:poincare-A}, it holds that 
\begin{equation}\label{etape:Bernstein}
\P_{\mu}\PAR{ p_{T,h}(y)-p_h(y) <-p_{\min}/4} \le \exp\PAR{-\frac{T p_{\min}^2}{32(\sigma^2+Mp_{\max}/(4\kappa_{\min}\lambda_1))}}.\end{equation}
We first bound $M$. 
     As $\int_{\cM}g\dd \mu=0$, and by assumption on the kernel $K$, we deduce
\begin{align*}
M^2&\le  \int (K_h(x,y)-p_h(y))^2\mu(\dd x)=\int K_h(x,y)^2\mu(\dd x)-p_h(y)^2\le \int K_h(x,y)^2\mu(\dd x)\\
&\le p_{\max}\|K_h\|_\infty^2 \int_{\|x-y\|\le h}\dd x\le  4p_{\max}\|K\|_\infty^2 h^{-2d} c_1h^d,
\end{align*}
where we use the fact the geodesic distance is equivalent to the Euclidean distance (see \cite[Proposition 2]{Trillos2020}), \Cref{lem:integral}, and \Cref{lem:properties-Kh}-$(iii)$.
    Hence, $M\le c_2\|K\|_\infty \sqrt{p_{\max}} h^{-d/2}$ for $h$ small enough.
    
We then bound $\sigma^2$: according to \Cref{lemma: var-lem 3} and \Cref{lemma: var-lem2},  and introducing the Green operator $G$ defined in \eqref{eq:green-op},
    \begin{align*}
        \sigma^2&\le \frac{2}{p_{\min}\kappa_{\min}} \int_\cM |(-\Delta)^{-1/2}(gp)|^2\dd x\\
        &= \frac{2}{p_{\min}\kappa_{\min}} \int_\cM gp \,G(gp)\dd x\\
        &\le  \frac{2 }{p_{\min}\kappa_{\min}} \PAR{p_{\max}\|g\|_\infty\int_{\|x-y\|\le  h} |G(gp)(x)|\dd x+ 4p_{\max}^3\int_{\|x-y\|\ge  h} |G(p)(x)|\dd x},
    \end{align*}
    because $g(x)=p_h(y)$ on $\BRA{\NRM{x-y}\ge h}$ and $p_h\le 2p_{\max}$.
    It remains to bound the right-hand side. First, according to \Cref{lem:properties-Kh} and \Cref{lem:ph-to-p}, for $h\le 1$,
    \begin{equation}\label{etape:borne_g}\|g\|_\infty \le 2p_{\max}+ 2\|K\|_\infty h^{-d}\le c_3 h^{-d}.
    \end{equation}
    Second, $|G(gp)(x)|\le |G(K_h(\cdot,y)p)(x)|+ p_h(y)|G(p)(x)|$, with
    \[ |G(p)(x)|\le p_{\max}\int_\cM |G(x,z)|\dd z\le p_{\max}c_4 \]
    according to \Cref{prop:greenfunction}-$(iv)$ and \Cref{lem:integral}. Then, using again the equivalence between the geodesic and the Euclidean distances, as the function $K_h(\cdot,y)p$ is supported on a geodesic ball $\cB(y,c_5 h)$, according to \Cref{lem:estimation-green} and \Cref{lem:properties-Kh}-$(iii)$, for all $x\in \cB(y,c_5 h)$, 
    \[ |G(K_h(\cdot,y)p)(x)|\le \kappa_1\|K_h(\cdot,y)p\|_\infty h^2 \le 2\kappa_1p_{\max}\|K\|_\infty h^{2-d}.\]
    In total, as $\int_{\|x-y\|\le  h} \dd x\le c_6h^d$ according to \Cref{lem:integral}, \eqref{etape:borne_g} and by assumption $\vol(\cM)=1$, it holds that  for $h\le 1$,
    \begin{align*}
    \sigma^2&\le \frac{2 p_{\max}^2}{p_{\min}\kappa_{\min}} \PAR{2c_3 c_6\PAR{\kappa_1\NRM{K}_\infty h^{2-d}+p_{\max}c_4}+ 4p_{\max}^2c_4}
    \\
    &\le c_8
h^{2-d}.
    \end{align*}

As $d\ge 4$, $M$ is smaller than $c_2\|K\|_\infty \sqrt{p_{\max}}h^{2-d}$. 
As $p_h(y)\ge p_{\min}/2$ for $h\le h_c$, we have $ \P_{\mu}\PAR{ p_{T,h}(y)<p_{\min}/4}\le \P_{\mu}\PAR{ p_{T,h}(y)-p_h(y) <-p_{\min}/4} $ for such a value of $h$. We therefore obtain the desired result.
		\end{proof}

We then conclude with a standard union bound argument by using a covering of $\cM$.
	\begin{proposition}\label{prop:proba-invariante-unif}
 Assume that $d\ge 4$. There exists $h_0$ depending on $\cM$, $K$,  $p_{\min}$ and  the $\cC^1$-norm of $p$ such that  for any $T>0$ and $0<h\le h_0$, we have
\[
 \P_{\mu}\PAR{\inf_{y\in \cM}p_{T,h}(y)<p_{\min}/8 } \le c_1 h^{-d(d+1)} \exp\PAR{- c_0 T h^{d-2} }
\]
for some constant $c_1$ depending only on $p_{\min}$, $K$ and $\cM$, where $c_0$ is the constant of \Cref{prop:bernstein}.
	\end{proposition} 
\begin{proof}
		To have a uniform estimation of $p_{T,h}$, we will make use of the Lipschitz continuity of $K$ and the covering number $N_{\delta}(\cM)$ of $\cM$, for $\delta>0$, i.e. the smallest number $N$ such that there exists a subset $E$ of $N$ distinct points of $\cM$ such that  $\max_{y \in \cM} \min_{x\in E}\rho(x,y) \le \delta$.

 By \Cref{lem:properties-Kh}, the function $p_{T,h}$ is Lipschitz continuous with constant $2\mathrm{Lip}(K) h^{-d-1}$ as an average of Lipschitz continuous functions. 
Let $\delta>0$. We consider the covering number $N_{\delta}(\cM)$ of $\cM$.
By \cite[2.2A]{Gromov1981},  there are constants $c_{\cM}$ and $\delta_{\cM}$ depending only on $\cM$ such that for all $0<\delta\le \delta_{\cM}$,
\begin{equation}\label{eq:covering number of manifold}
N_{\delta}(\cM) \le c_{\cM}  \delta^{-d}.
\end{equation}
Consequently, for $h\le h_0$ (where $h_0$ is the constant of \Cref{prop:bernstein}), if $y_1,\dots,y_{N_{\delta}}$ is a minimal $\delta$-covering of $\cM$,
 \begin{align*}
 \P_{\mu}\PAR{ \inf_{ y\in\cM} p_{T,h}(y) <p_{\min}/8  }&\le\sum_{i=1}^{N_{\delta }(\cM)}
\P_{\mu}\PAR{ \exists y\in\cB(y_i,\delta ):p_{T,h}(y) <p_{\min}/8   }\\
&\le\sum_{i=1}^{N_{\delta }(\cM)}
\P_{\mu}\PAR{ p_{T,h}(y_i) <p_{\min}/8+2\mathrm{Lip}(K) h^{-d-1} \delta  }.
\end{align*}
Choose $\delta =\frac{p_{\min}h^{d+1}} {16\mathrm{Lip}(K)}$, which is smaller than $\delta_{\cM}$ as long as $h$ is small enough with respect to $\cM$ and $K$, as $p_{\min}\le 1$. 
By Proposition~\ref{prop:bernstein} and Equation \eqref{eq:covering number of manifold}, we easily deduce
 \begin{align*}
  \P_{\mu}(\inf_{y\in \cM}p_{T,h}(y)<p_{\min}/8)&\le c_{K,\cM}p_{\min}^{-d}  h^{-d(d+1)}\exp\PAR{- c_0 T h^{d-2} },
\end{align*}
and the result is proved.
\end{proof}

\subsection{Proof  of \Cref{prop:variance} }\label{sec:conclusion-proof-mainA}

We first give a last useful result related to the diffusion $(X_t)_{t\ge 0}$ with generator $\cA$ in relation with the operators.
\medskip

\textbf{Notation.} Given a space $E$, for any function $f: E \times E \rightarrow \R$ and any operator $\cJ: D \subset \mathbb{R}^E \rightarrow   \mathbb{R}^E$, we define $\cJ_1f$ when the operator is applied to the first variable of $f$ and $\cJ_2f$ when it is applied to the second variable of $f$:
\begin{align*}\label{eq:notation-operator}
	\cJ_1f(x,y):=\PAR{\cJ f(.,y)}(x)\quad \text{ and }\quad \cJ_2f(x,y):=\PAR{\cJ f(x,.)}(y).
\end{align*}

\begin{proposition}\label{prop:1/2-op}
	Let $R \in L^2( \mu \otimes \dd y)$ be a function such that for all $(x,y) \in \cM^2$, $R(x,\cdot ) \in L^2_0(\dd y)$ and $R(\cdot, y) \in L^2_0(\mu)$ . Then, when the initial distribution of the diffusion $(X_t)_{t\ge 0}$ is its invariant measure $\mu$,  we have
	\begin{align*}
		\E_{\mu}\SBRA{\int_\cM\ABS{\PAR{-\Delta}^{-1/2}\PAR{\frac{1}{T}\int_0^T R(X_s, \cdot )\dd s}}^2\dd y}&\le \frac{2}{T}\iint_{\cM^2}\ABS{\PAR{-\cA}_1^{-1/2}\PAR{-\Delta}_2^{-1/2} R }^2\dd\mu\dd y.
	\end{align*}
\end{proposition}
\begin{proof}
	Denote by $\widehat{R}$ the function $(-\Delta)^{-1/2}_2R$. We observe that $\widehat{R} \in L^2( \mu \otimes \dd y)$ and for any $y$, $\widehat{R}(\cdot,y ) \in L^2_0(\mu )$. Hence, by applying  Lemma \ref{lemma: var-lem2}, we have the following sequence of equalities:
	\begin{align*}
		& \E_{\mu}\SBRA{\int_\cM\ABS{\PAR{-\Delta}^{-1/2}\PAR{\frac{1}{T}\int_0^T R(X_s, \cdot )\dd s}}^2\dd y} =\E_{\mu}\SBRA{\int_\cM\ABS{\frac{1}{T}\int_0^T \PAR{-\Delta}^{-1/2} R(X_s, \cdot )\dd s}^2\dd y}\\
		&= \int_\cM \E_{\mu}\SBRA{\ABS{\frac{1}{T}\int_0^T \widehat{R}(X_s, y )\dd s}^2}\dd y \le \frac{2}{T}\int_{\cM} \PAR{\int_{\cM}\ABS{\PAR{-\cA}_1^{-1/2}\widehat{R} }^2\dd\mu}\dd y.
	\end{align*}
	Therefore, the proposition is proved.
\end{proof}

Using the decomposition~\eqref{eq:E+E^c}, we can now prove Proposition~\ref{prop:variance}. We first use the estimate of the probability of $E_{T,h}^c$ given in \Cref{prop:proba-invariante-unif}, and then give an explicit estimate of the variance term  on the event $E_{T,h}$ for a diffusion $(X_t)_{t\ge 0}$ with generator $\cA$ starting from its invariant measure $\mu$.

\begin{proof}[Proof of Proposition~\ref{prop:variance}] \

According to \eqref{eq:control-P(E^c)} and \Cref{prop:proba-invariante-unif}, the probability of the event $E^c_{T,h}$ is negligible. Indeed, {when $d\ge 4$ and for $T$ large enough such that $Th^{d-2}> \frac{d^2+4}{c_0}\ln(T)$,} the second term in the decomposition \eqref{eq:E+E^c} is smaller than $\frac{p^2_{\max}}{p^2_{\min}}\|K\|_\infty^2 h^{4-d}T^{-1}$.
Furthermore, it is equal to $0$ when $K$ is nonnegative. It remains to bound the first term.

On the event $E_{T,h}$, both $\mu_h$ and $\mu_{T,h}$ are probability measures with respective density functions $p_{T,h}$ and $p_h$. Furthermore, for $h<h_c$, $p_h\ge p_{\min}/2$. Hence, by \Cref{lem:wang}, we have: 
	\begin{equation}
		 \label{ineq:  variance inequality 1}
		\W_2^2(p_{T,h},p_h)\le \frac{8}{p_{\min}} \int_\cM\ABS{\PAR{-\Delta}^{-1/2}\PAR{p_{T,h}-p_h}}^2\dd x.
	\end{equation}

	Now, consider the function $R_h(x,y)= K_h(x,y) -\int_\cM K_h(z,y) \mu (\dd z)$. First, we observe  that $$p_{T,h}(y) -p_h(y)=  \frac{1}{T} \int_0^T  R_h(X_s,y) \dd s.$$
	As $R_h$ is continuous on a compact manifold $\cM$, we have $R_h \in L^2( \mu \otimes \dd y)$. 
	Thus, for each $x,y \in \cM$, $R_h(x, \cdot ) \in L_0^2(\dd y)$ and $R_h(\cdot, y) \in L_0^2(\mu)$.  Therefore, due to Proposition \ref{prop:1/2-op},
	\begin{equation} \label{ineq:  variance inequality 3}
		\E_\mu  \SBRA{\int_\cM \ABS{\PAR{-\Delta}^{-1/2}\PAR{p_{T,h}-p_h}}^2\dd y} \le  \frac{2}{T}\iint_{\cM^2}\ABS{\PAR{-\cA}_1^{-1/2}\PAR{-\Delta}_2^{-1/2} R_h }^2\dd \mu \dd y.
	\end{equation}

Besides, by Lemma \ref{lemma: var-lem 3}, we have:
	\begin{equation} \label{ineq:  variance inequality 4}
\iint_{\cM^2}\ABS{\PAR{-\cA}_1^{-1/2}\PAR{-\Delta}_2^{-1/2} R_h }^2\dd\mu\dd y \le  \frac{1}{p_{\min}\kappa_{\min}} \iint_{\cM^2}\ABS{\PAR{-\Delta}_1^{-1/2}(M_p)_1\PAR{-\Delta}_2^{-1/2} R_h }^2\dd x \dd y,
\end{equation}
	where $M_p: L^2_0(\mu) \rightarrow L^2_0(\dd x)$ is the bounded multiplication operator $f \mapsto pf$. 
\\	
Therefore, after Inequalities \eqref{ineq:  variance inequality 1},	\eqref{ineq:  variance inequality 3}, \eqref{ineq:  variance inequality 4}, we have:
$$ \E_{\mu}\SBRA{\W_2( \mu_{T,h} ,\mu_h )^2\ind_{E_{T,h}}} \le \frac{{16}}{p_{\min}^2\kappa_{\min}T}\iint_{\cM^2}\ABS{\PAR{-\Delta}_1^{-1/2}(M_p)_1\PAR{-\Delta}_2^{-1/2} R_h }^2\dd x\dd y.$$
Note that $M_p$ and $\Delta^{-1/2}$ are bounded operators, which implies $(M_p)_1$ and $(\Delta^{-1/2})_2$ are commutative. In other words, $ (M_p)_1\PAR{-\Delta}_2^{-1/2} R_h=  \PAR{-\Delta}_2^{-1/2}(M_p)_1 R_h$, which means
 \[
     \E_{\mu}\SBRA{\W_2^2( \mu_{T,h} ,\mu_h )\ind_{E_{T,h}}} \le
  \frac{{16}}{p_{\min}^2\kappa_{\min}T }  \iint_{\cM^2}\ABS{\PAR{-\Delta}_1^{-1/2}\PAR{-\Delta}_2^{-1/2} S_h }^2\dd x  \dd y,
 \] 
 where $S_h(x,y)= p(x) R_h(x,y)  \in L^2( \dd x \otimes \dd y)$.
 
 Recall that $(\lambda_i)_{i\ge0}$ and $(\phi_i)_{i\ge 0}$ are respectively the eigenvalues and the eigenfunctions of $(-\Delta)$. As $p$ is upper bounded, $\forall x\in\cM$ $S_h(x,.) \in L_0^2( \dd y)$, $\forall y\in\cM$ $S_h(.,y) \in L_0^2( \dd x)$ and $S_h \in L_0^2( \dd x \otimes \dd y)$. Thus, there are $(\alpha_{i,j}(h) )_{i,j \ge 0}$ such that  $S_h$ has the following decomposition (with $\alpha_{i,0}=\alpha_{0,j}=0$): for all $x,y\in \cM$
 $$S_h(x,y) = \sum_{i,j \ge 1} \alpha_{ij}(h) \phi_i(x) \phi_j(y),$$
with $\alpha_{ij}(h)=\iint_{\cM^2} S_h(x,y)\phi_i(x)\phi_j(y)\dd x\dd y$. Consequently, for all $x,y\in \cM$
 \begin{align*}
 	\PAR{-\Delta}_1^{-1/2}\PAR{-\Delta}_2^{-1/2} S_h (x,y)&=  \sum_{i,j \ge 1} \frac{\alpha_{ij}(h)}{\sqrt{\lambda_i \lambda_j}} \phi_i(x) \phi_j(y), \\
 	\PAR{-\Delta}_1^{-1} S_h(x,y) & =  \sum_{i,j \ge 1} \frac{\alpha_{ij}(h)}{\lambda_i } \phi_i(x) \phi_j(y), \\
 	 	\PAR{-\Delta}_2^{-1} S_h(x,y) & =  \sum_{i,j \ge 1} \frac{\alpha_{ij}(h)}{\lambda_j } \phi_i(x) \phi_j(y). 
 \end{align*}
 Therefore,
 \begin{align*}
 	&\iint_{\cM^2}\ABS{\PAR{-\Delta}_1^{-1/2}\PAR{-\Delta}_2^{-1/2} S_h }^2\dd x  \dd y = \sum_{i,j \ge 1} \frac{\alpha_{ij}^2(h)}{\lambda_i \lambda_j}  \\ 	
 	=& \iint_{\cM^2}(\Delta_1^{-1}S_h)(\Delta_2^{-1} S_h) \dd x  \dd y = \iint_{\cM^2}(G_1S_h)(x,y)(G_2S_h)(x,y) \dd x  \dd y,
 \end{align*}
where $G$ is the Green function of $\Delta$ introduced in Section~\ref{sec:green}.
	{Denoting by $\widetilde{K}_h(x,y)=p(x)K_h(x,y)$, we have $S_h(x,y)=\widetilde{K}_h(x,y)- p(x)p_h(y)$, and thus, 
\begin{multline}\label{eq:G1SG2S}
(G_1S_h)(x,y)(G_2S_h)(x,y)=\\
(G_1\widetilde{K}_h)(x,y)(G_2\widetilde{K}_h)(x,y)-(G_1\widetilde{K}_h)(x,y)p(x)Gp_h(y)\\
-(G_2\widetilde{K}_h)(x,y)p_h(y)Gp(x)
+p(x)Gp(x)p_h(y)Gp_h(y).
\end{multline}}
Let $d\ge 5$. Using the fact that $K$ has compact support and that the geodesic and Euclidean distances are equivalent, the functions $x\mapsto K_h(x,y)$ and $y\mapsto K_h(x,y)$ are supported on a geodesic ball of radius $c_1h$ for some $c_1>0$ depending only on $\cM$. Hence,  \Cref{lem:estimation-green} implies that for $h$ sufficiently small and  for all $x\in\cM$,
	\begin{align}
		&\int_{\cM}\ABS{G_1{\widetilde{K}}_h(x,y)G_2{\widetilde{K}}_h(x,y)}\dd y
        \notag\\
   &\le {c\,}\NRM{{\widetilde{K}}_h}_\infty^2\PAR{h^4 \int_{\cB(x,2c_1 h)}\dd y+h^{2d}\int_{\cM\setminus\cB(x,2c_1 h)}\rho(x,y)^{4-2d}\dd y}
   \notag\\
		 &\le {c\,} p_{\max}^2 \| K_h\|_{\infty}^2\PAR{h^4 \int_{\cB(x,2h)}\dd y+h^{2d}\int_{\cM\setminus\cB(x,2h)}\rho(x,y)^{4-2d}\dd y}
		 \notag\\
   &\le {c\,} p_{\max}^2 \| K\|_{\infty}^2 h^{-2d} \PAR{h^{4+d} +h^{2d}h^{4-d}}
		\notag\\
		 &\le {c\,} p_{\max}^2 \| K\|_{\infty}^2 h^{-2d}\times h^{4+d},\label{eq:bound1}
	\end{align}
where we also use \Cref{lem:integral} and \Cref{lem:properties-Kh}{, and the constant $c$ depends only on $\cM$ and can change from line to line.  
Besides, using similar computations, we have, for $h\le 1$  sufficiently small
\begin{align}
 &\iint_{\cM^2}\ABS{(G_1\widetilde{K}_h)(x,y)p(x)Gp_h(y)}  \dd x\dd y
 \notag\\
 &\le c\, p_{\max}^2 \| K_h\|_{\infty}\int_\cM\PAR{h^2\int_{\cB(y,2h)}\dd x+h^d\int_{\cM\setminus \cB(y,2h)}\rho(x,y)^{2-d}\dd x}\ABS{Gp_h(y)}\dd y 
 \notag\\
 &\le  c\, p_{\max}^2 \| K_h\|_{\infty}h^{d}\int_\cM\ABS{Gp_h(y)}\dd y\le  c\, p_{\max}^2 \| K\|_{\infty}\int_\cM\ABS{Gp_h(y)}\dd y, 
 \label{eq:bound2}\\
  &  \iint_{\cM^2}\ABS{(G_2\widetilde{K}_h)(x,y)p_h(y)Gp(x)} \dd x\dd y \notag \\
&\le c\,p_{\max}\NRM{p_h}_\infty \| K_h\|_{\infty}\int_\cM\PAR{h^2\int_{\cB(x,2h)}\dd y+h^{d}\int_{\cM\setminus \cB(x,2h)}\rho(x,y)^{2-d}\dd y}\ABS{Gp(x)}\dd x
  \notag\\
  &\le c\, p_{\max}\NRM{p_h}_\infty \| K\|_{\infty}\int_{\cM}\ABS{Gp(x)}\dd x,
\label{eq:bound3}
\\
  & \iint_{\cM}\ABS{p(x)Gp(x)p_h(y)Gp_h(y)}\dd x  \dd y
   \le p_{\max} \NRM{p_h}_\infty\int_\cM\ABS{Gp(x)}\dd x\int_\cM \ABS{Gp_h(y)}  \dd y.
\label{eq:bound4}
\end{align}
By Lemma~\ref{lem:ph-to-p}, $(p_h)$ converges to $p$ uniformly on $\cM$ when $h\to 0$. Consequently, for $h$ sufficiently small, $\NRM{p_h}_\infty\le 2p_{\max}$. Finally, by Proposition~\ref{prop:greenfunction} $(iv)$ and \eqref{eq:bound geodesic distance on a small ball} in Lemma~\ref{lem:integral},
 we have, for $h$ sufficiently small,
 \begin{align}
 \int_\cM\ABS{Gp(x)}\dd x&\le c\, p_{\max}\PAR{\int_\cM\int_{\cB(x,h)}\rho(x,z)^{2-d}\dd z\dd x+\int_\cM\int_{\cM\setminus\cB(x,h)}\rho(x,z)^{2-d}\dd z\dd x}
 \le c\, p_{\max}\label{eq:bound5}
 \\
 \int_\cM\ABS{Gp_h(y)}\dd y&\le c\, \NRM{p_h}_\infty \le c\, p_{\max}.
 \label{eq:bound6}
 \end{align}
 Combining Equations \eqref{eq:G1SG2S}-\eqref{eq:bound6}, we deduce that, for $h$ sufficiently small,
 \begin{align*}
 \iint_{\cM^2}(G_1S_h)(x,y)(G_2S_h)(x,y)\dd x\dd y\le cp_{\max}^2\NRM{K}_\infty^2 h^{4-d}.
 \end{align*}
}
This proves the proposition when $d\ge 5$. The computations for $d\le 4$ are similar, and left  to the reader.
\end{proof}

\section{Transition to a general initial measure}\label{sec:general-distribution}

In the previous section, we have obtained a control of the variance term $\W_2^2(\widehat{\mu}_{T,h},\mu_h)$  when the stochastic process $(X_t)_{t\ge 0}$ starts from its invariant measure $\mu$. In this section, we explain how to extend the result to an initial measure of type Dirac measure $\delta_x$  (which will imply the result for any initial measure).
The main idea is to use the  ultracontractivity of the diffusion $(X_t)_{t\ge 0}$.

Let us first introduce this notion.
\begin{lemma} \cite[Theorem 3.5.5.]{Wang2014}  \label{lemma: ultracontractivity}
	The semigroup $(P_t)$ associated to the operator $\cA$
 is ultracontractive. In other words, for each $t>0$, there is a minimal positive value $c_t>0$, such that for any bounded measurable function $f$, we have
	\begin{equation}   \|  P_tf \|_{\infty} \le c_t\|f\|_{L^1(\mu)}  .\label{eq:constante_ct}
	\end{equation}
\end{lemma}
 In \Cref{sec:ultracontracticity}, an explicit form of the ultracontractivity term $c_t$ is given.
  We denote by $u_\cA=c_1$ the ultracontractivity constant at time $t=1$.

\begin{remark} Wang  only considers operators of the form $\Delta+\nabla p$ in \cite{Wang2014}. However, as explained in \Cref{rk:Apq} (see \Cref{sec:ultracontracticity} for more details), for any $\mathcal{C}^2$ second order elliptic differential operator $\mathcal{A}$ on a smooth manifold $\mathcal{M}$ such that $\mathcal{A}$ is symmetric with respect to the measure $\mu(\dd x)=p(x) \dd x$, there is always a $\mathcal{C}^2$-Riemannian metric $\tilde{\mathbf{g}}$ on $\mathcal{M}$ such that $\mathcal{A}=\tilde{\Delta}+ \tilde{\nabla}p,$ where $\tilde{\Delta}$ and  $\tilde{\nabla}$ are respectively the Laplacian and the gradient operator of $(\cM, \tilde{\mathbf{g}})$. Hence, \cite[Theorem 3.5.5.]{Wang2014} can readily be applied.
\end{remark}

\begin{remark}
    In \cite{Wang2014}, Wang actually defines the ultra-contractivity of a semigroup $(P_t)_{t\ge 0}$  in a slightly different way: for any $t>0$, there should exist $c_t>0$ such that for any measurable bounded function $f$, $ \|  P_tf \|_{\infty} \le c_t\NRM{f}_{L^2(\mu)}$. However,  Wang's definition  implies \eqref{eq:constante_ct} in our setting with $\cM$ compact. Assume that $(P_t)_{t\ge 0}$ is ultra-contractive in Wang's sense, then for $t>0$ and for any measurable bounded function $f$,
    \[
        \NRM{P_tf}_{\infty}\le c_{t/2}\NRM{P_{t/2}f}_{L^2(\mu)}
\le \PAR{c_{t/2}}^2\|f\|_{L^1(\mu)}
    \]
     because $\NRM{T}_{2\to \infty}=\NRM{T}_{1\to 2}$ for a symmetric operator $T$. 
\end{remark}

We now use the ultracontractivity constant $u_\cA$ to control the distance $\mathbb{E}_{x} \left[ \cW_2^2( \widehat{\mu}_{T,h},\mu_h) \right]$.
\begin{proposition}\label{prop:ultracontractivite}
	Let $x\in\cM$. There exist constants $c$, $h_0$ depending only on $\cM$, and on $\cM$ and $K$ respectively such that for any $x\in \cM$ and any $T>1$, and any $h\le h_0$, we have that:
	\begin{multline}\label{etape2}
	    \mathbb{E}_{x} \left[ \cW_2^2( \widehat{\mu}_{T,h},\mu_h) \right] \le c  \frac{ \| K\|_{\infty}}{T} + 2\mathrm{diam}(\cM)^2\P_{x}( E^c_{T,h})+ 2u_\cA\mathrm{diam}(\cM)^2\P_\mu(E^c_{T,h}) \\ + 2u_\cA \E_\mu\SBRA{\cW_2^2( \widehat{\mu}_{T,h} ,\mu_h)}.
     \end{multline}
\end{proposition}
{Notice that the last term in the right hand side of \eqref{etape2} is the expectation starting from the initial distribution $\mu$, thanks to ultracontractivity. This term has been controlled in Proposition \ref{prop:variance}.}

\begin{proof}	
	Consider functions $F(x)= \E_x \SBRA{\cW_2^2( \widehat{\mu}_{T,h},\mu_h)}$ and $H(x) = \P_x(E^c_{T,h})$. 
	Note that the shifted process $(\widetilde{X}_t)_{t\ge 0}= (X_{t+1})_{t\ge 0}$ is still a diffusion process with generator $\cA$. Therefore, if we consider the following shifted quantities
 \begin{align*}
     &\widetilde{p}_{T,h}(y) =\frac{1}{T}\int_0^T K_h(X_{s+1},y)\dd s,\\
     & \widetilde{\mu}_{T,h} = \begin{cases}
     \widetilde{p}_{T,h} \dd x &\text{ if }\widetilde{p}_{T,h} \dd x\text{ is positive measure}\\[0.1cm]
     \delta_{x_0} &\text{ otherwise,}
     \end{cases} \\
     & \widetilde{E}_{T,h}=\BRA{ \widetilde{p}_{T,h} \ge 0},
 \end{align*}
 we have
	\begin{align*}
		&\mathbb{E}_{x}\left[ \cW^2_2( \widetilde{\mu}_{T,h},\mu_h) \right] = \mathbb{E}_{x} \left[\mathbb{E}_{X_1}\SBRA{ \cW^2_2(  \widehat{\mu}_{T,h},\mu_h)} \right]= \E_{x}\SBRA{F(X_1)}= P_1F(x),
  \\
  &\P_{x} (\widetilde{E}^c_{T,h} ) = \E_{x}\SBRA{ \P_x( E^c_{T,h})}= \E_{x}\SBRA{H(X_1)}=  P_1H(x).
\end{align*}
Thus, by ultracontractivity of the process $(X_t)_{t\ge 0}$, at time $t=1$,
\begin{align}
	&\mathbb{E}_{x}\left[ \cW^2_2(  \widetilde{\mu}_{T,h},\mu_h) \right] \le \NRM{P_1F}_{\infty} \le u_\cA \int_{\cM} F(y) \mu( \dd y) =u_\cA \E_{\mu} \SBRA{\cW^2_2( \widehat{\mu}_{T,h},\mu_h)},\label{ineq: contractivity 1}\\
 &\P_{x}( \widetilde{E}^c_{T,h}) \le u_\cA \P_{\mu}(E^c_{T,h}), \label{ineq: contractivity 2}
\end{align}where $u_\cA$ has been defined after Lemma \ref{lemma: ultracontractivity}.\\

 Recall that the Wasserstein distance is controlled by the total variation distance \cite[Theorem 6.15]{villani2009optimal}.
Besides, on the event $E_{T,h}\cap \widetilde{E}_{T,h}$, $p_{T,h} \dd x$ and $\widetilde{p}_{T,h} \dd x$ are both positive measures. Hence, on $E_{T,h}\cap \widetilde{E}_{T,h}$, by Lemma~\ref{lem:properties-Kh}, when $h$ is sufficiently small depending on $\cM$ and $K$,
\begin{align*}
&\cW^2_2( \widehat{\mu}_{T,h}, \widetilde{\mu}_{T,h})\le 
2\mathrm{diam}(\cM)^2 \int_{\cM} |p_{T,h}(y)-\widetilde{p}_{T,h}(y)| \dd y \\ 
&\le 2\mathrm{diam}(\cM)^2 \frac{1}{T} \int_{\cM} \PAR{ \int_0^1 |K_h(X_s,y)| \dd s +\int_T^{T+1} |K_h(X_s,y)| \dd s }  \dd y \\
&=  2\mathrm{diam}(\cM)^2 \frac{1}{T}  \PAR{ \int_0^1 \PAR{ \int_{\cM} |K_h(X_s,y)| \dd y} \dd s +\int_T^{T+1} \PAR{ \int_{\cM} |K_h(X_s,y)| \dd y } \dd s }  \\
&\le  2\mathrm{diam}(\cM)^2\frac{1}{T} 2\| K\|_{\infty} h^{-d} \PAR{\int_0^1\int_{\cM}\ind_{\NRM{X_s-y}\le h}\dd y+\int_T^{T+1} \int_{\cM}\ind_{\NRM{X_s-y}\le h} \dd y\dd s} 
\\
&\le  c_0 \frac{\| K\|_{\infty}}{T},
\end{align*}
where $c_0$ is a constant depending on $\cM$ and we used \Cref{lem:integral} for the last inequality  and the equivalence between the Euclidean and geodesic distances.

Thus, by \eqref{ineq: contractivity 2},  for $h$ sufficiently small,
\begin{align}
    \E_{x}\SBRA{ \cW^2_2( \widehat{\mu}_{T,h}, \widetilde{\mu}_{T,h})} &\le c_0  \frac{\| K\|_{\infty}}{T}  + \mathrm{diam}(\cM)^2\P_{x}\PAR{E^c_{T,h} \cup \widetilde{E}^c_{T,h}} \nonumber \\
    &\le c_0 \frac{\| K\|_{\infty}}{T} +\mathrm{diam}(\cM)^2 \P_{x}( E^c_{T,h})+u_\cA\mathrm{diam}(\cM)^2 \P_\mu(E^c_{T,h}) .\label{ineq: contractivity 3}
\end{align}

Consequently, by triangular inequality, \eqref{ineq: contractivity 1}, and \eqref{ineq: contractivity 3}, we deduce that there exists a constant $c$ depending only on $\cM$ such that for $h$ sufficiently small, 
\begin{align*}
    & \frac{1}{2}\E_{x}\SBRA{ \cW^2_2(\widehat{\mu}_{T,h}, \mu_h )} \le  \E_{x}\SBRA{\cW^2_2(\widehat{\mu}_{T,h}, \widetilde{\mu}_{T,h})} + \E_{x}\SBRA{ \cW^2_2(\widetilde{\mu}_{T,h}, \mu_h )}\\
    \le& c_0 \frac{ \| K\|_{\infty}}{T} + \mathrm{diam}(\cM)^2\P_{x}( E^c_{T,h})+u_\cA\mathrm{diam}(\cM)^2\P_\mu(E^c_{T,h})  +u_\cA \E_\mu\SBRA{\cW_2^2\PAR{\widehat{\mu}_{T,h} ,\mu_h}}, 
\end{align*}
which is the desired conclusion.
\end{proof}

Hence, using \Cref{prop:variance} and \Cref{prop:proba-invariante-unif}, it only remains to bound $ \P_{x}( E^c_{T,h})$. This probability is $0$ if $K$ is nonnegative. Otherwise we assume that $d\ge 4$ and that $Th^{d}\ge c_0$ for a constant $c_0$ to fix.

	\begin{lemma}\label{prop:proba-any-distribution}
 Assume that $d\ge 4$. Let $h_0$, $c_0$ and $c_1$ be  the constants of  \Cref{prop:proba-invariante-unif}. Then, for any $T\ge 2$, if $0<h\le h_0$ and $Th^d>32\|K\|_\infty/p_{\min}$, we have
	\[ \P_{x}\PAR{ E_{T,h}^c }  \le u_\cA c_1h^{-d(d+1)}\exp\PAR{- \frac{c_0}{2}Th^{d-2}}.\]
\end{lemma}

\begin{proof}{When $T\ge 2$, we have that $1/2\le (T-1)/T\le 1$.
Hence we have the following observation for all $y \in \cM$ and $T\ge 2$:}
\begin{align*}
	p_{T,h}(y)= & \frac{1}{T}\int_0^1 K_h(X_s,y)\dd s +\frac{T-1}{T} \times \frac{1}{T-1}  \int_0^{T-1}  K_h(X_{s+1},y)\dd s \\
	\ge &  - \frac{1}{T} \int_0^1 \| K_h\|_{\infty} \dd s + \frac{{C}}{T-1}  \int_0^{T-1}  K_h(X_{s+1},y)\dd s,
\end{align*}{where $C=1/2$ if $\int_0^{T-1}  K_h(X_{s+1},y)\dd s>0$ or $C=1$ otherwise. }Thus,
	\begin{align*}
		\P_{x}\PAR{ \inf_{ y\in\cM}p_{T,h}(y) <0 } & \le  \P_{x} \PAR{   \inf_{y \in \cM } \PAR{ \frac{{C}}{T-1}  \int_0^{T-1} K_h(X_{s+1},y)\dd s } <\frac{\|K_h\|_{\infty}}{T }}  \\
		 &\le \E_{x} \PAR{\P_{X_1}\PAR{ \inf_{ y\in\cM}p_{(T-1),h}(y) < \frac{2\|K_h\|_{\infty}}{ T }}}\\
&\le u_\cA \P_{\mu}\PAR{ \inf_{ y\in\cM}p_{(T-1),h}(y) < \frac{2\|K_h\|_{\infty}}{  T }},
	\end{align*}
where we used  ultracontractivity at the last step. By Lemma~\ref{lem:properties-Kh}, we have $\|K_h\|_{\infty}\le 2\|K\|_\infty h^{-d}$ for $h$ small enough. Hence, if $Th^d>32\|K\|_\infty/  p_{\min}$, then $2\|K_h\|_{\infty}/T <p_{\min}/8$, and we obtain thanks to \Cref{prop:proba-invariante-unif} that
\begin{align*}
    \P_{x}\PAR{E_{T,h}^c}&\le u_\cA \P_{\mu}\PAR{ \inf_{ y\in\cM}p_{(T-1),h}(y) < p_{\min}/8}\\
&\le u_\cA c_{1}h^{-d(d+1)}\exp(-c_0(T-1)h^{d-2}).
\end{align*}
As $T-1\ge T/2$, for $T\ge 2$, the result holds.
\end{proof}

We put together all the estimations obtained so far to conclude. Consider $d\ge 4$ and $Th^{d}\ge c'_0$, where the constant $c'_0$ is chosen so that $Th^d>32\|K\|_\infty/p_{\min}$ and $Th^{d-2}$ is large enough with respect to $\ln(T)$, so that the upper bound on $\P_{x}\PAR{ E_{T,h}^c } $ given in \Cref{prop:proba-any-distribution} is negligible. The proof is similar when $d<4$ and $K$ is nonnegative.
By Proposition~\ref{prop:ultracontractivite}, \Cref{prop:variance}, \Cref{prop:proba-any-distribution} and \Cref{prop:proba-invariante-unif}, for any $x\in\cM$, for $h$ sufficiently small, and $T\ge 2$, there are constants $c,\tilde c>0$ such that
\begin{align*}
&\mathbb{E}_{x} \left[ \cW_2^2( \widehat{\mu}_{T,h},\mu_h) \right] 
\\
&\le c  \frac{ \| K\|_{\infty}}{T} + 2\mathrm{diam}(\cM)^2\P_{x}( E^c_{T,h})+ 2u_\cA\mathrm{diam}(\cM)^2\P_\mu(E^c_{T,h})  + c u_\cA \frac{p_{\max}^2}{p_{\min}^2} \|K\|_\infty^2 \frac{h^{4-d}}T\\
&\le \tilde{c} \, u_\cA\frac{p_{\max}^2}{p_{\min}^2} \|K\|_\infty^2 \frac{h^{4-d}}T.
\end{align*}

\section{Minimax lower bound}\label{sec:minimax}

The proof of the minimax lower bound (\Cref{prop:minimax}) relies crucially on the computation of the Kullback-Leibler divergence between the law of two diffusion processes $(X_t)$ and $(X'_t)$ on $\cM$. Recall that for two probability measures $\P$ and $\Q$, the Kullback-Leibler divergence $\mathrm{KL}(\P \|\Q)$ is defined as 
\begin{equation}
    \mathrm{KL}(\P \|\Q) = \int \ln \left( \frac{\dd \P}{\dd \Q}\right) \dd \P
\end{equation}
whenever $\P$ is absolutely continuous with respect to $\Q$.

Recall that we have denoted, in \Cref{sec:main-results}, by $\cP_{T,\ell}=\cP_{T,\ell}(\kappa_{\min},p_{\min},p_{\max},u_{\max},L)$ the class of all the laws of diffusion paths $(X_t)_{t\in [0,T]}$ on $\cM$ with arbitrary initial distribution, generator $\cA$ satisfying \Cref{assump:A} with constant $\kappa_{\min}$ and having an ultracontractivity constant smaller than $u_{\max}$, and invariant measure $\mu$  having a $\cC^2$ positive $p$ density on $\cM$ with a Sobolev norm $\|p\|_{H^\ell(\cM)} \le L$ and a $\cC^1$-norm smaller than $L$, satisfying $p_{\min}\le p\le p_{\max}$.

Our minimax lower bound follows from an application of Assouad's lemma, see \cite[Theorem 2.12]{Tsybakov2009}. 
\begin{lemma}[Assouad's lemma]
Consider a statistical model $\cP$.   Let $J>0$ be an integer and consider a subfamily $\{\P_\tau\}_\tau\subset \cP$ indexed by $\tau \in \{\pm 1\}^J$. Define the Hamming distance $d_H(\tau,\tau')=\sum_{j=1}^J \mathbf{1}_{\{\tau_j\neq \tau'_j\}}$ on the hypercube $\{\pm 1\}^J$. Assume that there exists $A$ such that for every $\tau,\tau'\in  \{\pm 1\}^J$, $\W_2(\mu(\P_\tau),\mu(\P_{\tau'}))\ge A d_H(\tau,\tau')$ and that whenever $d_H(\tau,\tau')=1$, $\mathrm{KL}(\P_{\tau}\| \P_{\tau'})\le 1/2$. Then, the minimax risk over $\cP$ defined in \eqref{eq:minimax} satisfies
\begin{equation}
    \cR(\cP)\ge \frac{AJ}{4}.
\end{equation}
\end{lemma}

To apply Assouad's lemma, we build an appropriate family $\{\mu_\tau\}$  of probability measures on $\cM$ indexed by the hypercube $\{\pm 1\}^J$ by 
perturbating the uniform measure by small bumps. We let $p_\tau$ be the density of $\mu_\tau$  and $\P_\tau$ be the law of a diffusion path $(X_t)_{t\in [0,T]}$ with initial distribution $\mu_\tau$ and  generator $\cA_{p_\tau q}$ (defined in \eqref{eq:A}) with $q\equiv 1$. In that case, the associated \textit{carré du champ} is equal to $\Gamma(f,f)=|\nabla f|^2$, so that we can take $\kappa_{\min}=1$.

\begin{lemma}[Existence of bumps]\label{lem:construction_bump}
Let $\ell\ge 0$ be an integer. There exist constants $\kappa$, $\eps_0>0$ depending only on $\ell$ and $\cM$ such that
    for all $x\in \cM$ and all $0\le \eps \le \eps_0$, there exists a smooth function $\phi_{x,\eps}:\cM\to \R$ supported on $\cB(x,\eps)$, with $\int_\cM \phi_{x,\eps} \dd y=0$, $\int_{\cM}\phi_{x,\eps}^2 \dd y = \eps^d $ and for all integers $0\le i\le \ell$,
    \begin{equation}
        \|\phi_{x,\eps}\|_{\cC^i(\cM)}\le \kappa\eps^{-i}.
    \end{equation}
\end{lemma}
\begin{proof}
    {Fix $x\in \cM$. Consider a smooth chart $\Psi$ around $x$ defined on the unit ball in $\R^d$, with $\Psi(0)=x$. As $\cM$ is smooth and compact, $\Psi$ can be choosen as being a bi-Lipschitz map with Lipschitz constant $L$ not depending on $x$. In particular, for $\eps$ small enough, the chart $\Psi:\cB_{\dR^d}(0,L^{-1}\eps)\to \cM$ is a diffeomorphism whose image is contained in $\cB(x,\eps)$. Let $\phi_0:\R^d\to [-1,1]$ be a smooth function of integral $0$, with support included in $\cB_{\dR^d}(0,L^{-1})$, equal to $1$ on $\cB_{\dR^d}(0,L^{-1}/3)$. Let $J\Psi:\cB_{\dR^d}(0,L^{-1}\eps)\to \R$ be the Jacobian of $\Psi$. We define for $u \in \cB_{\dR^d}(0,\eps L^{-1})$ and $y=\Psi(u)$
    \[ \phi_{x,\eps}(y)= C_{x,\eps}\frac{\phi_0(u/\eps)}{J\Psi(u)}, \]
    and $\phi_{x,\eps}$ is $0$ outside $\Psi(\cB_{\dR^d}(0,\eps L^{-1}))$. By construction, $\int_{\cM}\phi_{x,\eps} \dd y = C_{x,\eps} \int_{\cB_{\dR^d}(0,\eps L^{-1})} \phi_0(u/\eps)\dd u=0$. We choose $C_{x,\eps}$ so that $\int_{\cM}\phi_{x,\eps}^2 \dd y=\eps^d$. It remains to bound the $\cC^i$-norm of $\phi_{x,\eps}$ for $0\le i\le \ell$.  Note that as $\cM$ is compact and smooth, all the coefficients of the metric and its inverse have bounded $\cC^k$-norms for all $k\ge 0$. This implies that for $\eps$ small enough, the chart $\Psi$ can be chosen so that the $\cC^{\ell}$-norm of the function $J\Psi: \cB_{\dR^d}(0,\eps/2)\to \R$ is uniformly bounded (uniformly with respect to the base point $x$), with $J\Psi(u)\ge c_0$ for all $u\in \cB_{\dR^d}(0,\eps L^{-1})$ and some constant $c_0$ depending on the manifold $\cM$. In particular, as $\phi_0$ is equal to $1$ on $\cB_{\dR^d}(0,L^{-1}/3)$, this implies that the constant $C_{x,\eps}$ is larger than $c_1$ for some constant $c_1$ depending only on $\cM$. }The smoothness of $\phi_0$ and of $J\Psi$ then imply the desired controls on the $\cC^i$-norm of $\phi_{x,\eps}$ (applying Leibniz formula for the derivative of a product).
\end{proof}

Fix $0< \eps\le \eps_0$ and consider a set of points $x_1,\dots,x_J\in \cM$ that are all at least $2\eps$ apart. Note that by a simple covering argument, the number $J$ can be chosen to be of order $\eps^{-d}$. For $j=1,\dots,J$, write $\phi_j=\phi_{x_j,\eps}$. Assume without loss of generality that the manifold $\cM$ has unit volume. For $\tau\in \{\pm 1\}^J$, consider the probability measure
\[ p_\tau = 1+ \frac{v}{2\kappa} \sum_{j=1}^J \tau_j \phi_j.\]
for some $0\le v\le \eps^\ell$. Note that according to \Cref{lem:construction_bump}, $p_\tau$ satisfies $1/2\le p_{\tau}\le 3/2$ and integrates to $1$.  According to \Cref{lem:construction_bump}, the $\cC^\ell$-norm of $p_\tau$ is smaller than $1/2$ (and therefore so is its Sobolev norm).

Using \Cref{rk:u_max}, 
we see that for a choice of $p_{\min}\le 1/2$, $p_{\max}\ge 3/2$, $\kappa_{\min}\le 1$, $L\ge 1/2$, all the associated measures $\mu_\tau$ are in the statistical model $\cP_{T,\ell}=\cP_{T,\ell}(\kappa_{\min},p_{\min},p_{\max},u_{\max},L)$ for some $u_{\max}$ large enough, as long as $\ell\ge 2$.

Let $\tau,\tau'\in \{\pm 1\}^J$. Consider the function $f= \sum_{j=1}^J \mathbf{1}_{\{\tau_j\neq \tau'_j\}}\tau_j \phi_j$. The function $f:\cM\to \R$ is  Lipschitz continuous with Lipschitz constant $\kappa\eps^{-1}$. 
As the 2-Wasserstein distance $\W_2$ is larger than the 1-Wasserstein distance $\W_1$, it holds by duality \cite[Particular Case 5.16]{villani2009optimal} that
\begin{align*}
\W_2(\mu_\tau,\mu_{\tau'})&\ge \W_1(\mu_\tau,\mu_{\tau'})\ge \frac{1}{\kappa\eps^{-1}} \left(\int_{\cM} fp_\tau \dd y -\int_{\cM} fp_{\tau'} \dd y\right) \\
&= \frac{v}{2\kappa^2\eps^{-1}}\sum_{j=1}^J \mathbf{1}_{\{\tau_j\neq \tau_{j'}\}} 2\tau_j\int_{\cM} \phi_jf \dd y.
\end{align*}
By construction of $f$ and as $\int_{\cM}\phi_j^2\dd y=\eps^d $, this quantity is equal to $\frac{v\eps^{d+1}}{\kappa^2} d_H(\tau,\tau')$, that is the first condition in Assouad's lemma holds for $A=v\eps^{d+1}/\kappa^2$.

The Kullback-Leibler divergence between $\P_{\tau}$ and $\P_{\tau'} $ can be controlled using Girsanov theorem. Indeed,  Girsanov theorem provides an explicit formula for the Kullback-Leibler divergence between $\P_{\tau}$ and $\P_{\tau'}$ (see \Cref{A:Kullback-Liebler}):
\begin{align*}
    \mathrm{KL}(\P_{\tau} || \P_{\tau'} ) &= \frac T4 \int_{\cM} \|\nabla \ln p_\tau - \nabla \ln p_{\tau'}\|^2 p_\tau^2 \dd y = \frac T4 \int_{\cM} \|\nabla p_\tau -\frac {p_\tau} {p_{\tau'}} \nabla p_{\tau'}\|^2 \dd y \\
    &\le  \frac T2 \int_{\cM} \|\nabla p_\tau - \nabla p_{\tau'}\|^2 \dd y+  \frac T2 \int_{\cM} \frac{\|\nabla p_{\tau'}\|^2}{p_{\tau'}^2}(p_\tau-p_{\tau'})^2\dd y.
\end{align*}
Consider $\tau,\tau'\in \{\pm 1\}^J$ with $d_H(\tau,\tau')=1$, and assume without loss of generality that the two vectors differ only by their first entry. 
Using the available bound on the  $\cC^1$-norm of $p_\tau$ and the fact that $p_{\tau}\ge 1/2$, we obtain that
\begin{align*}
    \mathrm{KL}(\P_{\tau} || \P_{\tau'} )
    &\le \frac T2 \int_{\cM} \|\nabla p_\tau - \nabla p_{\tau'}\|^2 \dd y  + \frac{v^2\eps^{-2}}{2}T \int_{\cM} ( p_\tau -  p_{\tau'})^2 \dd y \\
&\le \frac T2 \frac{v^2}{\kappa^2} \int_{\cB(x_1,\eps)} \|\nabla \phi_1\|^2  \dd y+ \frac{v^4\eps^{-2}}{2\kappa^2} T \int_{\cM} \phi_1^2 \dd y\\
&\le \frac T2 v^2\eps^{-2} \int_{\cB(x_1,\eps)}\dd y  + \frac{v^4}{2\kappa^2}
\eps^{-2+d} T  \lesssim T v^2\eps^{d-2},
\end{align*}
where we use at the last line that the volume of a ball of radius $\eps$ is of order $\eps^d$ for $\eps$ small enough (see \Cref{lem:integral}), and $v\in[0,\eps^\ell]$. 

When $d\ge 5$, choose $\eps= cT^{-1/(2\ell+d-2)}$ and $v=\eps^\ell$. For $c$ small enough, $\mathrm{KL}(\P_{\tau}\|\P_{\tau'})\le 1/2$.
 By Assouad's lemma, and recalling that we can pick $J$ of order $\eps^{-d}$, we obtain that
 \begin{equation}
      \cR(\cP_{T,\ell})\ge \frac{AJ}{4} = \frac{\eps^{\ell+d+1}J}{4\kappa^2} \gtrsim \eps^{\ell+1}\gtrsim T^{-\frac{\ell+1}{2\ell+d-2}},
 \end{equation}
concluding the proof of \Cref{prop:minimax} in this case. For $d\le 4$, we let $\eps=   {\eps_0}$ and $v=cT^{-1/2}$ for $c$ small enough, so that $\mathrm{KL}(\P_{\tau} || \P_{\tau'} )\le 1/2$ and $J$ is of order $1$. Then, Assouad's lemma gives
 \begin{equation}
      \cR(\cP_{T,\ell})\ge \frac{AJ}{4} = \frac{v \eps^{d+1}J}{4\kappa^2} \gtrsim T^{-1/2}.
 \end{equation}
This concludes the proof. \hfill $\Box$
\appendix 
\renewcommand{\theequation}{\Alph{section}.\arabic{equation}}
\let \sappend=\section
\renewcommand{\section}{\setcounter{equation}{0}\sappend}

\section{Wasserstein distance and bias term estimates}
\subsection{About the Wasserstein distance's estimate of Peyre and Proof of Lemma \ref{lem:wang}}\label{appendix:peyre estimate}

We provide here a new proof of \cite[Theorem 2.1]{peyre2018comparison}, a result we need to prove Lemma \ref{lem:wang} at the end of this section. The proof given by Peyre indeed holds only under additional assumptions that were not detailed in his paper and that turn out to be unnecessary.

\begin{theorem}[Theorem 2.1 in \cite{peyre2018comparison}]\label{theorem:peyre_appendix}
    Given a compact Riemannian manifold $\cM$.
    Then, for all probability measures $\mu$ and $\nu$ on $\cM$, we have:
    $$
        \cW_2(\mu,\nu) \le 2\| \mu-\nu\|_{H^{-1}(\mu)},
    $$
    where
    $$
        \| \mu-\nu\|_{H^{-1}(\mu)}:= \sup_{g:  \int_\cM | \nabla g|^2 \dd \mu \le 1} \PAR{ \int_\cM g\dd (\mu-\nu)}.
    $$
    Moreover,
    \begin{equation*}
        \cW_2(\mu,\nu) \le 2\sup_{\substack{g \in Lip(\cM): \\ \int_\cM | \nabla g|^2 \dd \mu \le 1}} \PAR{ \int_\cM g\dd (\mu-\nu)},
    \end{equation*}
    where $\text{Lip}(\cM)$ denotes the space of all Lipschitz continuous functions on $\cM$.
\end{theorem}
\begin{proof}
 Consider the Hamilton-Jacobi semigroup $(Q_t)_{t > 0}$ on $\text{Lip}(\cM)$:
    \[
        Q_t \phi (x) := \inf_{y \in \cM} \left\{ \phi(y) + \frac{1}{2t} \rho(y, x)^2 \right\}, \quad t > 0, \ \phi \in \text{Lip}(\cM),
    \]
    where $\rho$ is the geodesic distance of $\cM$.
    From \cite[Theorem 2.5]{LottVillani2007}, for any $\phi$ continuous bounded, $Q_0 \phi := \lim_{t \to 0} Q_t \phi = \phi$, $\| \nabla Q_t \phi \|_\infty$ is bounded for all $t > 0$, and $Q_t \phi$ solves the Hamilton-Jacobi equation:
    \begin{equation*}
        \frac{\dd}{\dd t} Q_t \phi = -\frac{1}{2} |\nabla Q_t \phi|^2, \quad t > 0.
    \end{equation*}
    Following the ideas of the proof of  \cite[Corollary 2.7.3]{Wang2014} and using \cite[Proposition 1.3.1]{Wang2014} (see also \cite[Theorem 1]{Rachev85}), the Kantorovich dual formula implies that:
    $$
        \frac{1}{2}\cW_2^2(\mu,\nu) = \sup_{\phi \in  \text{Lip}(\cM)\text{ bounded}} \BRA{ \nu(Q_1\phi)- \mu(\phi)}.
    $$
    Consider the following curve on the space of measures $(\mu_t)_{0\le t \le 1}$ defined by
    $\mu_t:=(1-t)\mu +t\nu.$
    By taking the derivative along $t$, we have that for all $\phi$ continuous and bounded:
    $$
        \frac{\dd }{\dd t} \PAR{ \mu_t(Q_t \phi )} = -\frac{1}{2} \PAR{\int_\cM  |\nabla Q_t \phi|^2\dd \mu_t}+\int_\cM Q_t\phi\dd(\nu-\mu).
    $$
     We now analyze the above term. Let $t\in(0,1)$.\\
 If $ \int_\cM  |\nabla Q_t \phi|^2\dd \mu_t >0$, by the classical inequality $-\frac{1}{2}a^2+ab \le \frac{1}{2}b^2$, we have:
    \begin{multline*}
          -\frac{1}{2} \PAR{\int_\cM  |\nabla Q_t \phi|^2\dd \mu_t}+\int_\cM Q_t\phi\dd(\nu-\mu)
        \\
        \le
          \frac{1}{2} \frac{\PAR{\int_\cM Q_t\phi\dd(\nu-\mu)}^2}{\int_\cM  |\nabla Q_t \phi|^2\dd \mu_t} \le \frac{1}{2}\sup_{\substack{g \in \text{Lip}(\cM): \\ \int_\cM | \nabla g|^2 \dd \mu_t \le 1}} \PAR{ \int_\cM g\dd (\mu-\nu)}^2
    \end{multline*}
    If $ \int_\cM  |\nabla Q_t \phi|^2\dd \mu_t =0$ and $\int_\cM Q_t\phi\dd(\nu-\mu)=0$, clearly we have
    \begin{align*}
        -\frac{1}{2} \PAR{\int_\cM  |\nabla Q_t \phi|^2\dd \mu_t}+\int_\cM Q_t\phi\dd(\nu-\mu)
        \le  \frac{1}{2}\sup_{\substack{g \in \text{Lip}(\cM): \\ \int_\cM | \nabla g|^2 \dd \mu_t \le 1}} \PAR{ \int_\cM g\dd (\mu-\nu)}^2.
    \end{align*}
    If $ \int_\cM  |\nabla Q_t \phi|^2\dd \mu_t =0$ and $\int_\cM Q_t\phi\dd(\nu-\mu)\neq 0$, using the Lipschitz function $\alpha\, Q_t\phi$ on $\cM$ for any arbitrary $\alpha\in\dR\setminus \{0\}$, we remark that 
    \begin{align*}
       \sup_{\substack{g \in \text{Lip}(\cM): \\ \int_\cM | \nabla g|^2 \dd \mu_t \le 1}} \PAR{ \int_\cM g\dd (\mu-\nu)}^2=\infty,
    \end{align*}
and clearly
    \begin{align*}
        -\frac{1}{2} \PAR{\int_\cM  |\nabla Q_t \phi|^2\dd \mu_t}+\int_\cM Q_t\phi\dd(\nu-\mu)
        \le  \frac{1}{2}\sup_{\substack{g \in \text{Lip}(\cM): \\ \int_\cM | \nabla g|^2 \dd \mu_t \le 1}} \PAR{ \int_\cM g\dd (\mu-\nu)}^2.
    \end{align*}
    Hence, for all $t \in (0,1)$,
    $$
        \frac{\dd }{\dd t} \PAR{ \mu_t(Q_t \phi )}  \le \frac{1}{2} \sup_{\substack{g \in \text{Lip}(\cM):                                                 \\ \int_\cM | \nabla g|^2 \dd \mu_t \le 1}} \PAR{ \int_\cM g\dd (\mu-\nu)}^2,
    $$
 which implies
    $$
        \cW_2^2(\mu,\nu)  \le \int_{0}^{1} \sup_{\substack{g \in \text{Lip}(\cM):                                                 \\ \int_\cM | \nabla g|^2 \dd \mu_t \le 1}} \PAR{ \int_\cM g\dd (\mu-\nu)}^2 \ dt.
    $$
    Besides, $\mu_t \ge (1-t)\mu$ for all $ t \in (0,1)$. Thus, for all $g$ with $\int_\cM | \nabla g|^2 \dd \mu_t \le 1$,  we have $(1-t)\int_\cM | \nabla g|^2 \dd \mu \le 1$.
    This leads to the fact that for all  $0<t_0<t_1<1$, the following holds
    \begin{align*}
        \cW_2^2(\mu_{t_0},\mu_{t_1}) & \le \int_{0}^{1} \sup_{\substack{g \in \text{Lip}(\cM):              \\ (1-t_1)\int_\cM | \nabla g|^2 \dd \mu \le 1}} \PAR{ (t_1-t_0)\int_\cM g\dd (\mu-\nu)}^2\dd t
        \\
                                     & \le  \frac{(t_1-t_0)^2}{1-t_1}\sup_{\substack{g \in \text{Lip}(\cM): \\ \int_\cM | \nabla g|^2 \dd \mu \le 1}} \PAR{ \int_\cM g\dd (\mu-\nu)}^2.
    \end{align*}
 Hence, for any $n \in \mathbb{N^*}$ and $0< t_0<t_1<...<t_n<1$, we have:
    $$
        \cW_2( \mu_{t_0},\mu_{t_n}) \le \sum_{i=1}^n \cW_2( \mu_{t_{i-1}},\mu_{t_i}) \le \sum_{i=1}^n \frac{t_{i}-t_{i-1}}{\sqrt{1-t_i}} \sup_{\substack{g \in \text{Lip}(\cM): \\ \int_\cM | \nabla g|^2 \dd \mu \le 1}} \PAR{ \int_\cM g\dd (\mu-\nu)}
    $$
    Thus, for all $0<t_0<t_1<1$, by convergence of the Riemann sum,
    \begin{align*}
        \cW_2(\mu_{t_0},\mu_{t_1}) & \le  \PAR{ \int_{t_0}^{t_1} \frac{1}{\sqrt{1-t}}\dd t}  \sup_{\substack{g \in \text{Lip}(\cM): \\ \int_\cM | \nabla g|^2 \dd \mu \le 1}} \PAR{ \int_\cM g\dd (\mu-\nu)} \\
                                   & =    2\PAR{\sqrt{1-t_0}-\sqrt{1-t_1}}  \sup_{\substack{g \in \text{Lip}(\cM):                      \\ \int_\cM | \nabla g|^2 \dd \mu \le 1}} \PAR{ \int_\cM g\dd (\mu-\nu)}.
    \end{align*}
    By taking $t_0 \rightarrow 0$ and $t_1 \rightarrow 1$, we obtain that:
    \begin{align*}
        \cW_2(\mu,\nu) \le 2 \sup_{\substack{g \in \text{Lip}(\cM): \\ \int_\cM | \nabla g|^2 \dd \mu \le 1}} \PAR{ \int_\cM g\dd (\mu-\nu)},
    \end{align*}
    which is the desired conclusion.
\end{proof}

\begin{proof}[Proof of Lemma \ref{lem:wang}] 
Before proving the result, let us establish that for all $f \in L^2(\dd x)$,  
\begin{equation}\label{eq:pr_reviewer}\|f\|_{H^{-1}(\dd x)}^2:=\sup_{g:  \int_\cM | \nabla g|^2 \dd x \le 1} \PAR{ \int_\cM g f \dd x } =\int_\mathcal{M} |(-\Delta)^{-1/2}f|^2 \mathrm{d}x.
\end{equation}
This equality is a consequence of Definition 7.24 p.202 in \cite{rudin1991} for example, or can be read from the developments in  Section 5.5.2 of \cite{santambrogio2015}. We provide a short proof of \eqref{eq:pr_reviewer} below.\\

Let $f$ be a function with $\int f\dd x=0$. 
          Recall that $\NRM{f}_{H^{-1}(\mu)}=\sup\BRA{\ABS{\int fg\dd x}:\NRM{\nabla g}_{L^2(\dd x)}\le 1}$. 
              First, for any function $g \in \mathcal{C}^\infty(\cM)$, by Cauchy-Schwartz inequality,
              \begin{align*}
                       \left| \int_\cM fg \dd x \right|^2  =  & \left| \int_\cM ( (-\Delta)^{-1/2}f) ( (-\Delta)^{1/2}g)  \dd x\right|^2 \\
                  \le  &  \| (-\Delta)^{-1/2}f\|^2_{L^2(\dd x)} \times \| (-\Delta)^{1/2}g\|^2_{L^2(\dd x)}.
              \end{align*}
              Besides,
              $$
                  \| (-\Delta)^{1/2}g\|^2_{L^2(\dd x))} = \int_\mathcal{M}  ((-\Delta)^{1/2}g)((-\Delta)^{1/2}g)\dd x = -\int_\cM g \Delta g\dd x = \int_\cM |\nabla g|^2\dd x .
              $$
              Therefore, by definition of, $\| \cdot\|_{H^{-1}(\dd x)}$, we have:
              \begin{equation*}
                  \|f\|_{H^{-1}(\dd x)}^2 \le \int_\mathcal{M} |(-\Delta)^{-1/2}f|^2 \mathrm{d}x. 
              \end{equation*}
Note also that equality is attained in Cauchy-Schwartz inequality if $(-\Delta)^{-1/2}f = \lambda (-\Delta)^{1/2} g$ for some $\lambda>0$. Letting $g= \lambda (-\Delta)^{-1}f$ with $\lambda = \|\nabla (-\Delta)^{-1}f\|_{L^2}^{-1}$ gives \eqref{eq:pr_reviewer}.\\

Let us now prove Lemma \ref{lem:wang}. From Theorem \ref{theorem:peyre_appendix} above, with $\mu=f_1 \dd x$, $\nu= f_2 \dd x$ and $f_1$ bounded below by $f_{\min}$,
\begin{align*}\cW_2\big(f_1\dd x,f_2\dd x\big)\leq &  2\NRM{f_1-f_2}_{H^{-1}(f_1\dd x)}\\
\leq &  2 \sup_{g:  \int_\cM | \nabla g|^2  f_1 \dd x \le 1} \PAR{ \int_\cM g (f_1-f_2) \dd x}\\
\leq &  2 \sup_{g:  \int_\cM | \nabla g|^2  f_{\min} \dd x \le 1} \PAR{ \int_\cM g (f_1-f_2) \dd x}\\
\leq &  2 f_{\min}^{-1/2} \sup_{g:  \int_\cM | \nabla g|^2   \dd x \le 1} \PAR{  \int_\cM g (f_1-f_2) \dd x}\\
= & 2 f_{\min}^{-1/2} \sqrt{\int_{\cM} \big|(-\Delta)^{-1/2} (f_1-f_2)\big|^2 \dd x}
\end{align*}
where we have used \eqref{eq:pr_reviewer} for the last equality. This concludes the proof.
\end{proof}

\subsection{Control of the bias term}\label{A:bias}
In kernel density estimation, variance terms can be controlled with minimal assumptions on the kernel $K$ (say, boundedness), whereas choosing a kernel having specific properties is required to control  bias terms \cite{Tsybakov2009}. The situation is not different for the estimation of $\mu$: we are able to control the variance term $\mathbb{E}_x\SBRA{\W_2^2(\widehat{\mu}_{T,h},\mu_h)}$ with few assumptions on $K$ (see \Cref{thm:mainA}). However, controlling the bias requires the kernel $K$ to be of sufficiently high order in the following sense.

\begin{definition}[Order of a kernel]\label{def:kernel}
 For a multi-index  $\alpha =(\alpha^1,...,\alpha^d) \in \mathbb{Z}_{+}^d$, we denote $|\alpha|:=\alpha^1+\cdots+\alpha^d$, $z^\alpha=\prod_{j=1}^dz_j^{\alpha_j}$ 
 and $\partial^\alpha \mathbf{K}$ the partial derivative of $\mathbf{K}$ in the direction $\alpha$.
 
 A function $\mathbf{K}: \mathbb{R}^d \rightarrow \mathbb{R}$ is called kernel of order $r$\index{order of a kernel} if the function is of class $\cC^r$, and   $\mathbf{K}$ satisfies 
	\[
 \int_{\dR^d}\partial^\alpha \mathbf{K}(z)z^{\tilde\alpha}\dd z=0,
    \]
    for any multi-index $\alpha, \tilde\alpha$ such that $ | \alpha | < r$ and $\ABS{\tilde\alpha}<r+| \alpha |$, with $\ABS{\alpha}>0$ when $\tilde\alpha=0$.
\end{definition}

When $K$ is of order larger than $\ell+1$, we obtain a tight control of the bias following \cite{Divol2022}.

\begin{proposition}[Bias term]\label{prop:bias}
Let $K$ be a kernel of order larger than $\ell+1$, and let $p$ be a density of class $\cC^2$ with a finite Sobolev norm $\|p\|_{_{H^\ell(\cM)}}$. Then, for $h$ small enough,
\begin{equation}
    \W_2^2(\mu_h,\mu) \le \frac{c\|p\|_{H^\ell(\cM)}^2}{p_{\min}^2} h^{2\ell+2},
\end{equation}
where $c$ depends on $\cM$ and $K$.
\end{proposition}

\begin{proof}
As $\mu$ has a lower bounded density, it holds according to \Cref{lem:wang} that 
\[
\W_2^2(\mu_h,\mu)\le 4p_{\min}^{-1}\int_{\cM}\ABS{ \PAR{-\Delta}^{-1/2}\PAR{p-p_h}}^2\dd x,
\]
whereas it is proved in \cite[Proposition 9]{Divol2022} that 
\[
\int_{\cM}\ABS{ \PAR{-\Delta}^{-1/2}\PAR{p-p_h}}^2\dd x\le c\|p\|_{H^\ell(\cM)}^2 h^{2\ell+2}
\]
when $K$ is a kernel of order larger than $\ell+1$, where $c$ depends on $\cM$ and $K$.
\end{proof}

\section{Some technical proofs}\label{A:proofs}

\subsection{Proof of Lemma~\ref{lem:properties-Kh}}\label{A:proof-properties-Kh}

\begin{enumerate}[label=(\roman*)]
\item 
The proof of the uniform convergence of $(h^{-d}\eta_h(.))_{h>0}$ is given in \cite[Lemma 10]{Divol2022} when the kernel $K$ has some regularity. 
We detail below another proof assuming only that $K$ is continuous with compact support in $[0,1]$.

Recall that since $\cM$ is compact, the Euclidean norm and the geodesic distance are equivalent: there is a constant $c_\cM>1$ such that $\forall (x,y)\in\cM^2$
\[
\NRM{x-y}\le \rho(x,y)\le c_\cM\NRM{x-y}.
\]

Consider a family $\cE_x: \B_{\R^d}(0,c_x) \rightarrow \cM$ of Riemannian normal parametrizations at $x\in\cM$ (see \cite[Section 3.1]{guerinnguyentran} or \cite[Proposition 5.24]{Lee2018}). 
Thanks to  \cite[Theorem 3.4]{guerinnguyentran} (for more details see \cite[Proposition 10.37]{Lee2018}),  there exists a constant $c_1>0$, such that all the parametrizations $(\cE_x)_{x\in\cM}$ have the same domain $\cB_{\dR^d}(0,c_1)$, and for any $h<\frac{c_1}{c_\cM}$ the Jacobian determinant $J_x$ of the change of variables $v=\cE_x^{-1}(y)$  is uniformly bounded in $x$.

Consequently,
\begin{align*}
h^{-d}\eta_h(x)&=h^{-d}\int_\cM K\PAR{\frac{\NRM{x-y}}{h}}\ind_{\NRM{x-y}\le h}\dd y\\
&=h^{-d}\int_{\cB_{\dR^d}(0,c_1)} K\PAR{\frac{\NRM{x - \cE_x(v)}}{h}}\ind_{\NRM{x - \cE_x(v)}\le h}J_x(v)\dd v\\
&=\int_{\cB_{\dR^d}\PAR{0,\frac{c_1}{h}}} K\PAR{\NRM{\frac{x - \cE_x(hv)}{h}}}\ind_{\NRM{\frac{x - \cE_x(hv)}{h}}\le 1}J_x(hv)\dd v.
\end{align*}

We know that $\rho(x,y)=\NRM{\cE^{-1}_x(y)}$ (see \cite[Theorem 3.2]{guerinnguyentran}). Then by equivalence of the distances, we have $\NRM{x-\cE_x(hv)}\ge \frac{\rho(x,\cE_x(hv))}{c_\cM}= h\frac{\NRM{v}}{c_\cM}$. Consequently,
\[
\BRA{\NRM{\frac{x - \cE_x(hv)}{h}}\le 1}= \BRA{\NRM{\frac{x - \cE_x(hv)}{h}}\le 1}\cap \BRA{\NRM{v}\le c_\cM }.
\]
Beside, by \cite[Theorem 3.4]{guerinnguyentran}, there exists $c_2>0$, independent of $x$, such that $\ABS{J_x(v)-1}\le c_2 \NRM{v}^2$  and $\NRM{x - \cE_x(v)-v}\le c_2\NRM{v}^2$ (since $\cE_x'(0)=Id$). 

We deduce that
 \begin{align*}
 &\ABS{K\PAR{\NRM{\frac{x - \cE_x(hv)}{h}}}}\ind_{\NRM{\frac{x - \cE_x(hv)}{h}}\le 1}J_x(hv)\ind_{\cB_{\dR^d}\PAR{0,\frac{c_1}{h}}}(v)\le \NRM{K}_\infty (1+c_2c_1^2)\ind_{\NRM{v}\le c_\cM},
 \end{align*}
where the upper-bound is an integrable function on $\dR^d$. 
We also remark that $(J_x(hv))_{h>0}$ converges to $1$ and $\PAR{\NRM{\frac{x - \cE_x(hv)}{h}}}_{h>0}$ converges to $\NRM{v}$ when $h\to 0$, both uniformly on $\cM$.
As $K$ is continuous on $\dR^d$ with compact support, the function is thus uniformly continuous on $\dR^d$. We then deduce that $\PAR{K\PAR{\NRM{\frac{x - \cE_x(hv)}{h}}}}_{h>0}$ converges to $K(\NRM{v})$ when $h\to 0$,  uniformly on $\cM$. 

By the dominated convergence theorem, we deduce that, when $h\to 0$, $\PAR{h^{-d}\eta_h}_{h>0}$ converges to $\int_{\dR^d}K(\NRM{v})\dd v=1$, uniformly on $\cM$.

\item By uniform convergence, $\eta_h>0$ when $h$ is small enough, so that $K_h$ is well-defined. The result then follows from a straightforward computation. 

\item  We note that the support of $K_h$ is included in $\BRA{(x,y)\in\cM^2:\NRM{x-y}\le h}$, and
\[
\ABS{K_h(x,y)}
\le  \frac{1}{\ABS{h^{-d}\eta_h(x)}}\NRM{K}_\infty h^{-d}.
\]

There is uniform convergence on $\cM$ of $(h^{-d}\eta_h(.))_{h>0}$ to $1$ when $h\to 0$.
Then there exists $ h_c>0$ such that $\forall h< h_c$, $\forall (x,y)\in \cM^2$, we have $h^{-d}\eta_h(x)\ge \frac{1}{2}$ and 
\[
\ABS{K_h(x,y)}
\le  2\NRM{K}_\infty h^{-d}.
\]

\item Let $x\in \cM$ and $y,y'\in \cM$. Let $L$ be the Lipschitz constant of $K$. Then, by the triangle inequality when $h<h_c$, so that  $h^{-d}\eta_h(x)\ge \frac{1}{2}$, we have
\begin{align*}
    |K_h(x,y)-K_h(x,y')|&\le \frac{1}{\ABS{h^{-d}\eta_h(x)}} h^{-d} |K\PAR{\frac{\NRM{x-y}}{h}}-K\PAR{\frac{\NRM{x-y'}}{h}}|\\
    &\le  2Lh^{-d-1}\|y-y'\| \le 2Lh^{-d-1}\rho(y-y').
\end{align*}
\end{enumerate}

\begin{remark}\label{rk:K-geodesic-appendix}
   As mentioned in Remark~\ref{rk:K-geodesic},
    it is simpler to work with a kernel $\widetilde{K}_h$ based on the geodesic distance $\rho$, $\widetilde{K}_h(x,y):=\frac{1}{\tilde{\eta}_h(x)}K\PAR{\frac{\rho(x,y)}{h}}$, 
with $\widetilde{\eta}_h(x)=\int_\cM K\PAR{\frac{\rho(x,y)}{h}}\dd y $. In that case, we can easily prove, without regularity assumptions, that when $K$ is an integrable function with support in $[0,1]$ and 
$\int_{\dR^d}K\PAR{\NRM{v}}\dd v=1$,
there is a constant $\kappa>0$ such that
$\NRM{h^d\widetilde{\eta}_h-1}_{\infty}\le \kappa h^2$.

\medskip
Actually, using 
 the Riemannian normal parametrization $\cE_x$ at $x\in\cM$ and $\rho(x,y)=\NRM{\cE^{-1}_x(y)}$, as in the previous proof of Lemma~\ref{lem:properties-Kh}, we obtain for $x\in\cM$
\begin{align*}
    &\ABS{h^{-d}\widetilde{\eta}_h(x)-1}=\ABS{h^{-d}\int_\cM K\PAR{\frac{\rho(x,y)}{h}}\ind_{\rho(x,y)\le h}\dd y-1}\\
&=h^{-d}\ABS{\int_{\cB_{\dR^d}(0,c_1)} K\PAR{\frac{\NRM{v}}{h}}\ind_{\NRM{v}\le h}J_x(v)\dd v-\int_{\cB_{\dR^d}(0,h)}K\PAR{\frac{\NRM{v}}{h}}\dd v}\\
&\le h^{-d}\int_{\cB_{\dR^d}(0,h)}\ABS{K\PAR{\frac{\NRM{v}}{h}}}\ABS{J_x(v)-1}\dd v\\
&\le c_2h^{2}\int_{\cB_{\dR^d}(0,1)}\ABS{K\PAR{\NRM{v}}}\NRM{v}^2\dd v
\end{align*}
since $\ABS{J_x(v)-1}\le c_2 \NRM{v}^2$. 
\end{remark}

\subsection{Proof of Lemma \ref{lem:ph-to-p}}\label{A:proof-ph}

Using the same proof as in Appendix~\ref{A:proof-properties-Kh}, we first remark that uniformly in $x\in\cM$ the term $\PAR{h^{-d}\int_\cM\ind_{\NRM{x-y}\le h}\dd y}_{h>0}$ converges to $\vol\PAR{B_{\dR^d}(0,1)}$.
    Besides, by the triangular inequality,
    \begin{align*}
        & \ABS{p_h(x)-p(x)} = \ABS{ \int_{\cM} \frac{1}{\eta_h(y)} K\PAR{\frac{\|x-y \|}{h}}p(y)\dd y -p(x)}
        \\
             &  \le \ABS{ \int_{\cM} \PAR{\frac{1}{\eta_h(y)} -h^{-d}} K\PAR{\frac{\|x-y \|}{h}}p(y)\dd y } +\ABS{ \int_{\cM} h^{-d} K\PAR{\frac{\|x-y \|}{h}}(p(y)-p(x))\dd y } +
        \\
        & \qquad + \ABS{ \int_{\cM} h^{-d} K\PAR{\frac{\|x-y \|}{h}}p(x)\dd y -p(x)}
        \\
        &\le       \NRM{\frac{1}{h^{-d}\eta_h}-1}_\infty \| K\|_{\infty}\|p\|_{\infty}\int_\cM  h^{-d} 1_{\| x-y\| \le h}  \dd y  +\| K\|_{\infty}\mathrm{Lip}(p)\int_\cM h^{-d}  1_{\| x-y\| \le h}  \ \| x-y\| \ \dd y + 
        \\
        & \qquad +|h^{-d} \eta_h(x)-1| | p(x)|,
    \end{align*}
    where $\mathrm{Lip}(p)$ is the Lipschitz constant of $p$ with respect to the Euclidean distance, which is bounded up to a constant by the $\cC^1$-norm of $p$ (recall that the Euclidean and geodesic distances are equivalent). Hence, by the remark at the beginning of this proof, there is a constant $C>0$ such that for all $x\in \cM$,
    \begin{align}\label{eq:conv-unif-ph}
         \ABS{p_h(x)-p(x)} \le  &  C  \  \|p\|_{\cC^1}\ \PAR{ {\|K\|_\infty}\NRM{\frac{1}{h^{-d}\eta_h}-1}_\infty+{\|K\|_\infty}h+\NRM{h^{-d}\eta_h-1}_\infty},
    \end{align}
    where the constant $C$ depends only on the embedding $\cM \subset \mathbb{R}^m$ and where the parenthesis in the right term  converges to zero by \Cref{lem:properties-Kh}(i). Therefore, we deduce the uniform convergence of $(p_h)_{h>0}$ to $p$ on $\cM$, and the second result of the lemma is straightforward a consequence of Equation~\eqref{eq:conv-unif-ph}. We have the desired conclusion.

\subsection{Proof of Lemma~\ref{lem:estimation-green}}\label{A:estimation-green-proof} 
We will make use of the following estimates.
\begin{lemma}\label{lem:integral}
Let  $x\in \cM$ and $\alpha \ge d$. For all $h>0$ small enough (depending on $\cM$ and $\alpha$), it holds that 
\begin{equation}
\int_{\rho(x,y)\ge h} \rho(x,y)^{-\alpha}\dd y \le C_{\alpha} \begin{cases}
h^{d-\alpha} &\text{ if }\alpha> d\\
\ln(1/h)&\text{ if }\alpha= d.
\end{cases} 
\end{equation} 
{For $\alpha< d$},
\begin{equation}\label{eq:bound geodesic distance on a small ball}
\int_{\rho(x,y)\le h} \rho(x,y)^{-\alpha}\dd y \le C_{\alpha} 
h^{d-\alpha}.
\end{equation} 
\end{lemma}
Before giving the proof, let us note that since the geodesic distance and the Euclidean norm are equivalent on the compact manifold $\cM$ (see \cite[Proposition 2]{Trillos2020}), we have a similar result with the Euclidean norm instead of the geodesic distance in \Cref{lem:integral}.
\begin{proof}
    Using the change of variables $v=\cE_x^{-1}(y)$ as in the proof of Lemma~\ref{lem:properties-Kh} (see Section~\ref{A:proof-properties-Kh}), we have for $h_0>0$  a constant larger than $h$
    \begin{align*}
        \int_{\rho(x,y)\ge h}\rho(x,y)^{-\alpha}\dd y&\le \int_{h_0>\rho(x,y)\ge h}\rho(x,y)^{-\alpha}\dd y +c h_0^{-\alpha}.
    \end{align*}
    We pick $h_0$ small enough (independently of $x$) so that, by a change of variable
    \begin{align*}
        \int_{h_0>\rho(x,y)\ge h}\rho(x,y)^{-\alpha}\dd y&= \int_{h_0>\|v\|\ge h} \|v\|^{-\alpha} J_x(v)\dd v,
    \end{align*}
    where $J_x(v)$ is a Jacobian, that is bounded by $2$ uniformly in $x$ and in $v$ on the domain $\BRA{v:h\leq \NRM{v}<h_0}$ (see \cite[Theorem 3.4 ii.]{guerinnguyentran}, also the proof of Lemma~\ref{lem:properties-Kh}). We then conclude by computing the integral. The proof of the second statement is similar.
\end{proof}

We only prove \Cref{lem:estimation-green} in the case $d\ge 3$, the cases $d=1$ and $d=2$ being treated with minimal modifications. 
By Proposition~\ref{prop:greenfunction} $(iv)$ 
on the Green function $G$, for $d\ge 3$, there exists a constant $\kappa>0$ such that $\forall(x,y)\in\cM^2\setminus \mathrm{diag}(\cM)$,
\[
\ABS{G(x,y)}\le \kappa\rho(x,y)^{2-d}.
\]
Hence,
\begin{align*}
\ABS{  \PAR{Gf}(z)}&\le\int_{\cM}\ABS{G(z,y)}\ABS{f(y)}\dd y\le \NRM{f}_{\infty}\kappa\int_{\cM}\rho(z,y)^{2-d}\ind_{\rho(x,y)\le h}\dd y.
\end{align*}

When $\rho(x,z)>2h$, we have $\rho(x,y)\le\rho(x,z)/2$ and
\[
\rho(z,y)\ge \rho(x,z)-\rho(x,y)\ge\frac{1}{2} \rho(x,z).
\]
Then, as $2-d<0$, we obtain
\begin{align*}
\ABS{  \PAR{Gf}(z)}
&\le 2^{d-2}\NRM{f}_{\infty}\kappa\int_{\cM}\ind_{\rho(x,y)\le h}\dd y \ \rho(x,z)^{2-d}.
\end{align*}
According to \Cref{lem:integral}, 
\[ \ABS{  \PAR{Gf}(z)}\le C\NRM{f}_{\infty} h^d \rho(x,z)^{2-d}.\]
We now turn to estimates when $\rho(x,z)<2h$. Notice that 
\begin{align*}
\ABS{  \PAR{Gf}(z)}
&\le \NRM{f}_{\infty}\kappa\int_{\cM}\rho(z,y)^{2-d}\ind_{\rho(z,y)\le 3h}\dd y.
\end{align*}
The conclusion then also follows from \Cref{lem:integral}.

\subsection{Wasserstein distance between a measure and its convolution}

\begin{lemma}\label{lem:convolution_wass}
    Let $\nu$ be a probability measure supported on $\cM$. Let $K$ be a nonnegative kernel supported on $[0,1]$, with $\int_{\R^d}K(\|u\|)\dd u=1$. For $h>0$, let $\nu_h$ be the measure with density $q_h(x) = \int K_h(z,x)\dd \nu(z)$, where $K_h$ is defined in \eqref{eq:def-K_h}. 
   {Then, for all $h>0$,
        \begin{equation}
        \W_2^2(\nu,\nu_h)\le  c_0 h^2.
    \end{equation}}
\end{lemma}
\begin{proof}
    By the convexity of the Wasserstein distance, it holds that
    \begin{equation}
        \W_2^2(\nu,\nu_h)\le \int \W_2^2(\delta_x,(\delta_x)_h)\nu(\dd  x)
    \end{equation}
    Let $c_0$ be such that $\rho(x,y)\le c_0\|x-y\|$ for all $(x,y)\in \cM^2$ (recall that the geodesic distance and the Euclidean distance are equivalent). 
   {As the measure $(\delta_x)_h$ is supported on a Euclidean ball of radius $h$, we directly have that $\W_2^2(\delta_x,(\delta_x)_h) \le c_0^2 h^2$. This ends the proof.}
\end{proof}

\subsection{The ultracontractivity term}\label{sec:ultracontracticity}

We give in this section an explicit control of the ultracontractivity constant introduced at the beginning of \Cref{sec:general-distribution}.
To this aim we need to introduce the iterated \textit{carré du champ}.
\begin{definition} Given a differential operator $\mathcal{A}$, its iterated \textit{carré du champ} $\Gamma_2$ is defined as: 
    $$\Gamma_2(f,f) := \frac{1}{2}\SBRA{ \mathcal{A}\Gamma(f,f)-2\Gamma(\cA f,f)},$$
    where $\Gamma$ is the \textit{carré du champ} of $\cA$.
\end{definition}
\begin{lemma}\label{lemma: ultracontractivity2}
Let $p \in \cC^2(\cM)$ be a positive density on $\cM$ with respect to the volume measure $\dd x$, and $\cA$ a $\mathcal{C}^2$-elliptic second-order differential operator, which is symmetric with respect to the probability measure $\mu =p \mathrm{d}x$.
Then, there is a constant $\kappa$, which can be negative, such that: 
\begin{equation}
    \Gamma_2(f,f) \ge \kappa \Gamma(f,f),
    \label{eq: curvature-dimension condition estimation}
\end{equation}
In particular, the associated semigroup $\PAR{P_t}_{t\ge 0}$ satisfies,  for all $t>0$, 
\begin{equation}
\label{eq:constante_ct2}
\| P_tf\|_{\infty} \le \exp\SBRA{ \frac{\kappa \mathrm{diam}(\cM)^2}{2(e^{2\kappa t}-1)}} \times \|f\|_{L^2(\mu)}.
\end{equation}
\end{lemma}

\begin{proof}
    Because $\mathcal{A}$ is a $\mathcal{C}^2$-elliptic operator of second-order, there is a $\mathcal{C}^2$-metric  $\tilde{\mathbf{g}}$ on $\mathcal{M}$ such that $\Gamma(f,f)= \SCA{\tilde{\nabla}f,\tilde{\nabla}f}_{\tilde{\mathbf{g}}}$, where $\tilde{\nabla}$ is the gradient of the new metric (see \cite[eq. 1.3.3]{Hsu2002}). 
    Hence, due to the symmetry of $\mathcal{A}$ with respect to $\mu$, $\mathcal{A} =\tilde{\Delta} + \tilde{\nabla}{ \ln(\tilde{p})}$, where $\tilde{\Delta}$ is the Laplacian of the new metric and $\tilde{p}= \frac{\dd \mu}{\dd \vol_{\tilde{\mathbf{g}}}}$, where $ \vol_{\tilde{\mathbf{g}}}$ is the volume measure in the new metric. 
    
    Consequently, by \cite[eq. C.5.3]{Bakry2014} we have $\Gamma_2(f,f)=|\tilde{\nabla} \tilde{\nabla} f |^2+\PAR{\mathrm{Ricc}_{\tilde{\mathbf{g}}} -\tilde{\nabla}\tilde{\nabla} \ln(\tilde{p})}\PAR{\tilde{\nabla} f,\tilde{\nabla} f } $, where $\mathrm{Ricc}_{\tilde{\mathbf{g}}}$ denotes the Ricci tensor. 
    Therefore, as in \cite[eq C.6.3]{Bakry2014}, \eqref{eq: curvature-dimension condition estimation} is equivalent to
    \begin{equation}\label{eq: a bound for ultracontractivity}
        \mathrm{Ricc}_{\tilde{\mathbf{g}}} - \tilde{\nabla}\tilde{\nabla} \ln(\tilde{p}) \ge \kappa.
    \end{equation}
    Hence, from the compactness of $\cM$ and the $\mathcal{C}^2$ continuity  of both $\mathcal{A}$ and $p$,  we have the desired conclusion for the first part.
    
    For the second part, we have from the implication $(1) \rightarrow (3)$ of Theorem 2.3.3 in \cite{Wang2014} with $p=2$, that for any $x$, 
    \begin{multline*}
        |P_tf(x)|^2 = \int_\cM | P_tf(x)|^2  \mu(\mathrm{d} y) \le \int_\cM  P_t |f|^2(y) \exp \SBRA{ \frac{\kappa \mathrm{diam}(\cM)^2 }{e^{2\kappa t}-1}} \mu(\mathrm{d}y)=\\
         \exp \SBRA{ \frac{\kappa \mathrm{diam}(\cM)^2 }{e^{2\kappa t}-1}} \int_\cM  P_t |f|^2(y)  \mu(\mathrm{d}y) = \exp \SBRA{ \frac{\kappa \mathrm{diam}(\cM)^2 }{e^{2\kappa t }-1}} \| f\|^2_{L^2(\mu)},
    \end{multline*}
   since $\mu$ is the invariant measure of the underlying process. Therefore, we have the second inequality, which is the ultracontractivity of $\mathcal{A}$.
\end{proof}

\begin{remark}\label{rk:u_max}
We note that for the operator $\cA_{pq}$ defined by \eqref{eq:A} with $q\equiv 1$, there is no need to change the metric to obtain the result. Consequently, the constant $\kappa$ appearing in \Cref{lemma: ultracontractivity2} depends only on $\cM$, an upper bound on $\| p\|_{\cC^2(\cM)}$ and $p_{\min}$.  This enables the choice of a uniform ultracontractivity constant  $u_{\max}$  for the minimax lower bound in \Cref{sec:minimax}. 
\end{remark}

{\subsection{Essentially self-adjoint operator}\label{appendix:selfadjoint}
Although $\mathcal{A}$ is symmetric, this does not guarantee that its adjoint $\mathcal{A}^*$ is symmetric unless $\mathcal{A}$ is self-adjoint. The essential self-adjointness of $\mathcal{A}$ means that its closure $\overline{\mathcal{A}}$ is self-adjoint, ensuring that the operator has a unique self-adjoint extension. 
\begin{proposition}
    \label{prop:essential_self_adjointness}
    If $\mathcal{A}: \mathcal{C}^2_c(\mathcal{M}) \rightarrow \mathcal{C}_0(\mathcal{M}) \subset L^2(\mu)$ is symmetric and the measure $\mu$ has a strictly positive density on $\mathcal{M}$, then $\mathcal{A}$ is essentially self-adjoint when considered as an unbounded operator on $L^2(\mu)$.
\end{proposition}
\begin{proof}
    Since $\mathcal{A}$ is symmetric on the Hilbert space $H = L^2(\mu)$, we have
    \[
        \int_{\mathcal{M}} \mathcal{A} f \, \mathrm{d}\mu = 0 \quad \text{for all } f \in \mathcal{C}^2_c(\mathcal{M}).
    \]
    The closure $\overline{\mathcal{A}}: \mathcal{D}(\overline{\mathcal{A}}) \rightarrow H$ of $\mathcal{A}$ generates a semigroup $(P_t^{(2)})_{t \ge 0}$ on $H$, satisfying $P_t^{(2)} f = P_t f$ for all $f \in \mathcal{C}_c(\mathcal{M})$ (see~\cite[Theorem 1.5.38]{toan-thesis}).
\\
    The operator $\overline{\mathcal{A}}$ is symmetric, being the closure of a symmetric operator. Moreover, since $\overline{\mathcal{A}}$ generates $(P_t^{(2)})_{t \ge 0}$, we have $P_t^{(2)} f \in \mathcal{D}(\overline{\mathcal{A}})$ for all $t > 0$ and $f \in \mathcal{C}_c(\mathcal{M})$.
\\
    Consider the function
    \[
        F(s) = \langle P_s^{(2)} f, P_{t - s}^{(2)} g \rangle_{L^2(\mu)},
    \]
    for fixed $t > 0$ and $f, g \in \mathcal{C}_c(\mathcal{M})$. This function is constant in $s \in [0, t]$. Indeed, for $0 < s < t$, we compute by symmetry of $\cA$
    \[
        \frac{\mathrm{d}}{\mathrm{d} s} F(s) = \langle \overline{\mathcal{A}} P_s^{(2)} f, P_{t - s}^{(2)} g \rangle_{L^2(\mu)} - \langle P_s^{(2)} f, \overline{\mathcal{A}} P_{t - s}^{(2)} g \rangle_{L^2(\mu)} = 0.
    \]
    This implies that
    \[
        F(s) = F(0) = \langle f, P_t^{(2)} g \rangle_{L^2(\mu)} \quad \text{for all } s \in [0, t].
    \]
    Therefore, the semigroup $(P_t^{(2)})$ is self-adjoint, meaning that $\PAR{P_t^{(2)}}^* = P_t^{(2)}$. Consequently, the generator $\overline{\mathcal{A}}$ is self-adjoint, and since $\mathcal{A}$ is densely defined, it follows that $\mathcal{A}$ is essentially self-adjoint.
\end{proof}
}

\section{SDEs for the diffusions with generator  $\cA_{pq}$ }\label{A:SDE}

In this section, we give the SDEs satisfied by the diffusion processes of generator \eqref{eq:A} and \eqref{eq:L}. 
Recall that the Stratonovich integral (see \cite[page 82]{Protter04}) is defined as follows:
\begin{definition}[Stratonovich integral]
Let $X,Y$ be two continuous real-valued semimartingales. The Stratonovich integral of $Y$ with respect to $X$, denoted by $\int_0^tY_s\circ\dd X_s$, is defined by
\[
\int_0^tY_s\circ\dd X_s:=\int_0^t Y_s\dd X_s+\frac{1}{2}\SCA{Y,X}_t,
\]
where the first term is the Itô integral of $Y$ with respect to $X$ and $\SCA{.,.}$ is the bracket process (also known as the quadratic covariation process).
\end{definition}

With the Stratonovich integral, the classical Itô formula can then be written is the following way (see \cite[Theorems 20-21, pages 277-278]{Protter04}), for a continuous $d$-dimensional semimartingale $X$ and a function $f:\dR^d\to\dR$ of class $\cC^2$: $f(X)$ is a semimartingale and
\begin{equation}\label{chainrule}
f(X_t)-f(X_0)=\sum_{i=1}^d\int_0^t\frac{\partial f}{\partial x^{i}}(X_s)\circ\dd X^i_s.
\end{equation}Because of this chain rule, the Stratonovich integral is better fitted to differential calculus on manifolds than the usual Itô integral. \\

We recall that a vector field $V$ on a manifold $\cM$ is a family $\BRA{V(x)}_{x\in\cM}$ such that $\forall x\in\cM$, $V(x)\in T_x\cM$ (see for e.g. \cite[Chapter 4]{Lee03}). In local coordinates $(x^1,x^2,...,x^d)$, a smooth vector field $V$ can be represented as
\[
V(x)=\sum_{i=1}^d V^{i}(x)\left.\frac{\partial}{\partial x^{i}}\right\vert_x,
\]
where $V^1,\ldots,V^d$ are real smooth functions on the domain of the local coordinate system, and where $\BRA{\frac{\partial}{ \partial x^i}}_{1\le i\le d}$ denotes a basis of $T_x\cM$.

\begin{proposition}[Theorem 1.2.9 in \cite{Hsu2002}]\label{prop:calculgene}

Let $l\ge 1$. Consider the Stratonovich SDE
\begin{equation}\label{eq:Vf}
\dd X_t=\sum_{\alpha=1}^lV_\alpha(X_t)\circ \dd B^\alpha_t+V_0(X_t)\dd t
\end{equation}
where
$(V_\alpha)_{0\le \alpha\le l}$ are $\cC^2$ vector fields and $B=(B^\alpha)_{1\le \alpha\le l}$ is the standard $l$-dimensional Brownian motion. 
Then, there exists a unique strong solution to \eqref{eq:Vf} (up to explosion time) whose infinitesimal generator is
\[
\cA f(x)=\frac{1}{2}\sum_{\alpha=1}^l \PAR{V_\alpha^2  f}(x)+(V_0f)(x),
\]
where $\PAR{V_\alpha^2 f}(x):=\PAR{V_\alpha \PAR{V_\alpha f}}(x)$, and whose \textit{carré du champ} operator is given by $\Gamma(f,g)=\frac{1}{2}\sum_{\alpha=1}^l V_\alpha(f) V_\alpha(g)$.
\end{proposition}

Let $\BRA{e_\alpha}_{1\le \alpha\le m}$ be an orthonormal basis on $\dR^m$.
For each $x\in\cM$, we consider $P_\alpha(x)$ the orthogonal projection of $e_\alpha$ to $T_x\cM$. Let us note that $P_\alpha$ is a vector field on $\cM$.
In a local coordinate system $(x^1,x^2,...,x^d)$, 
\[
P_\alpha(x)=\sum_{i=1}^dP^{i}_\alpha(x)\left.{\partial\over \partial x^i}\right\vert_x.
\]

Then the Laplace-Beltrami operator satisfies (see \cite[Theorem 3.1.4]{Hsu2002})
\begin{equation}\label{eq:def-LapP}
\Delta=\sum_{\alpha=1}^m P^2_\alpha.
\end{equation}
Remark that for two real-valued functions of class $\cC^2$ on $\cM$, we have:
\begin{equation}\label{etape:calcul-gene-gradient}
\langle \nabla f,\nabla h\rangle = \sum_{\alpha=1}^m (P_\alpha f)(P_\alpha h).
\end{equation}
Then, as an application of Proposition \ref{prop:calculgene}, there exists a unique strong solution starting at $x\in\cM$ to the following SDE
\begin{align}\label{def:eds-A}
\dd X_t
&=\sum_{\alpha=1}^m\sqrt{2q(X_t)}P_\alpha(X_t)\circ \dd B^\alpha_t+\sum_{\alpha=1}^m\PAR{\frac{1}{2}\PAR{P_\alpha q}(X_t)+q\PAR{P_\alpha (\ln p)}(X_t)}\PAR{P_\alpha f}(X_t)
\end{align}
 for $(B^\alpha)_{1\le \alpha\le m}$ independent euclidean 1-dimensional Brownian motions, and whose infinitesimal generator is
\begin{equation}\label{genA-app}
\cA_{pq} f
 =q\Delta f+\SCA{q\nabla \ln (p q),\nabla f}. 
\end{equation}

In the particular case where $q\equiv 1$, we deduce
that the unique solution to the SDE
\begin{align}\label{eq:EDS-L}
\dd X_t 
&=\sqrt{2}\sum_{\alpha=1}^mP_\alpha(X_t)\circ \dd B^\alpha_t+ \sum_{\alpha=1}^mP_\alpha \PAR{\ln p}(X_t)P_\alpha (X_t)\dd t,
\end{align}
has the infinitesimal generator
\[
\cL f=\Delta f +\SCA{\nabla \ln p, \nabla f}.
\]

\begin{proof}[Proof of \eqref{genA-app}] To compute the infinitesimal generator of \eqref{def:eds-A}, we apply Proposition \ref{prop:calculgene} with 
$V_\alpha=\sqrt{2q}P_\alpha$ and $V_0=\sum_{\alpha=1}^m\PAR{\frac{1}{2}\PAR{P_\alpha q}+q\PAR{P_\alpha (\ln p)}}P_\alpha $.
From this proposition, we know that the generator of $X$ is, for a test function $f$ of class $\cC^2$,
\begin{align*}
   \cA_{pq} f(x)=\sum_{\alpha=1}^m\frac{1}{2}(V^2_\alpha f)(x)+(V_0 f)(x).
\end{align*}
We have:
\begin{align*}
    \PAR{V_\alpha^2 f}(x)
    &=2\sqrt{q}P_\alpha(\sqrt{q}\PAR{P_\alpha f})(x)
    =2q(x)\PAR{P_\alpha^2 f}(x)+\PAR{P_\alpha q}(x)\PAR{P_\alpha f}(x)
\end{align*}
Thus,
\begin{align*}
\cA_{pq} f(x)
&=q(x)\Delta f(x)+\sum_{\alpha=1}^m\PAR{P_\alpha q}(x)\PAR{P_\alpha f}(x)+q\sum_{\alpha=1}^m\PAR{P_\alpha (\ln p)}P_\alpha.
\end{align*}
We conclude thanks to \eqref{etape:calcul-gene-gradient}.
\end{proof}

At last, since $\cA_{pq}$ is self-adjoint  with respect to $\mu$, we deduce the following proposition.

\begin{proposition}
The measure $\mu=p(x) \dd x$ is invariant for $\cA_{pq}$.	
\end{proposition}

\section{Kullback-Leibler divergence of path space measures on manifolds}\label{A:Kullback-Liebler}

In this section, the operator $\cL$ will be denoted by $\cL_p$ to highlight its dependence with respect to the density  $p$ of the measure $\mu$.
	Let us denote by \(\P_{(p,T)}\) the probability measure on $\cC([0,T],\mathcal{M})$, given by the distribution of the diffusion 
with generator $\mathcal{L}_p$, defined for $f\in \cC^2(\cM)$ by: 
\[\mathcal{L}_pf = \Delta f + \langle \nabla f, \nabla \ln p \rangle,\]
and starting from its invariant measure $\mu=p\dd x$.

We will denote by $\E_{(p,T)}$ the expectation in the distribution \(\P_{(p,T)}\) and we define the Kullback-Leibler divergence as:
\begin{equation}
    \mathrm{KL}(\P_{(p,T)} || \P_{(q,T)} )= \E_{(p,T)}\SBRA{\ln \frac{\dd\P_{(p,T)}}{\dd\P_{(q,T)}}}
\end{equation}where $\dd\P_{(p,T)}/\dd\P_{(q,T)}$ stands for the density of \(\P_{(p,T)}\) with respect to the measure \(\P_{(q,T)}\).

	\begin{theorem}
		\label{theorem: KL for path space measures}
		For any two \(\mathcal{C}^1\) strictly positive probability densities \(p\) and \(q\) on \(\mathcal{M}\), 
\begin{equation}	
		\mathrm{KL}(\P_{(p,T)} || \P_{(q,T)} ) 
		= \frac{T}{4}  \int_\mathcal{M} \| \nabla \ln p-\nabla \ln q\|^2 p^2 \dd x.
\end{equation}		
			\end{theorem}

The proof of \Cref{theorem: KL for path space measures} relies crucially on Girsanov's theorem.

	\begin{proposition}[Girsanov's theorem for embedded manifolds]
		\label{proposition: Girsanov's theorem for embedded manifold}
		Consider two continuous tangent vector fields \(Z_1,Z_2\) that are tangent to \(\mathcal{M}\). Suppose that $(X_t)_{t\ge 0}$ satisfies the Stratonovich SDE:
		\[ \mathrm{d} X_t= Z_2(X_t)\mathrm{d}t + \sqrt{2}\sum_{\alpha=1}^m P_\alpha (X_t) \circ \mathrm{d}B^\alpha_t, \qquad X_0 = x_0 \in \mathcal{M}, \]
		where \(B=(B^\alpha)_{1\le \alpha\le m}\) is a \(m\)-dimensional standard Brownian motion under \(\mathbb{P}\), and \((P_\alpha)_{1\le \alpha\le m}\) is the tangent projection of the standard basis of \(\mathbb{R}^m\) on \(\mathcal{M}\) (see Section \ref{A:SDE}).\\
		Then, defining for any $T>0$,
		\begin{equation}\label{martingale-expo-annexe}\mathcal{E}_T=\exp \left( - \int_0^T \frac{1}{\sqrt{2}}(Z_2-Z_1)(X_t)\mathrm{d}B_t +\frac{1}{4} \int_0^T \|Z_1(X_t)-Z_2(X_t) \|^2_2 \mathrm{d} t\right),\end{equation}
		 it follows that:
		\[ \mathbb{E}[ \mathcal{E}_T]=1 \]
		and under \(\mathrm{d} \mathbb{Q}= \mathcal{E}_T \mathrm{d} \mathbb{P}\), $(X_t)_{t\ge 0}$ is a solution to the SDE:
		\[ \mathrm{d} X_t= Z_1(X_t)\mathrm{d}t + \sqrt{2}\sum_{\alpha=1}^m P_\alpha(X_t) \circ \mathrm{d}\widetilde{B}^\alpha_t, \qquad X_0 = x_0 \in \mathcal{M}, \]
		with $\widetilde{B}$ being a Brownian motion under \(\mathbb{Q}\).
	\end{proposition}
	
	\begin{proof}[Proof for Proposition \ref{proposition: Girsanov's theorem for embedded manifold}]
Let $T>0$.	The first part follows from the fact that the Itô integral \(\int_0^T (Z_2-Z_1)(X_t) \mathrm{d}B_t\) is a local martingale with bounded quadratic variation (since the \(Z_i\)'s are continuous and \(\mathcal{M}\) is compact, hence the \(Z_i\)s are bounded on \(\mathcal{M}\)).
	For the second part, by Girsanov's theorem in $\R^m$, the stochastic process \((\widetilde{B}_t)_{ 0\le t \le T}\) defined by:
	\[ \widetilde{B}_t = B_t + \int_0^T u(X_t) \mathrm{d}t, \]
	with $u(x)= \frac{1}{\sqrt{2}} (Z_2-Z_1)(x)$, is a standard Brownian motion under \(\mathbb{Q}\).
	Besides, the Stratonovich SDE for \(X_t\) can be rewritten as:
	\begin{align*}
	\mathrm{d} X_t & = Z_2(X_t)\mathrm{d}t + \sqrt{2}\sum_{\alpha=1}^m P_\alpha(X_t) \circ \mathrm{d}B^\alpha_t \\
	& = Z_2(X_t)\mathrm{d}t + \sqrt{2}\sum_{\alpha=1}^m P_\alpha(X_t) \circ \left(\mathrm{d}\widetilde{B}^\alpha_t - u^\alpha(X_t)\mathrm{d}t \right) \\
	& = \sqrt{2}\sum_{\alpha=1}^m P_\alpha(X_t) \circ \mathrm{d}\widetilde{B}^\alpha_t+(Z_2(X_t)-\sum_{\alpha=1}^m \sqrt{2}u^\alpha (X_t) P_\alpha (X_t)) \mathrm{d} t \\
	& = Z_1(X_t)\ \mathrm{d}t + \sqrt{2}\sum_{\alpha=1}^m P_\alpha (X_t) \circ \mathrm{d}\widetilde{B}^\alpha_t,
	\end{align*}
	hence, we imply the desired conclusion.
		\end{proof}
	
	Let us now prove Theorem \ref{theorem: KL for path space measures}.

 \begin{proof}[Proof of Theorem \ref{theorem: KL for path space measures}]

 Let $(X_t)_{t\ge 0}$ be a solution to the SDE:
	\[ \mathrm{d} X_t= \nabla\ln p(X_t)\mathrm{d}t + \sqrt{2}\sum_{\alpha=1}^m P_\alpha \circ \mathrm{d}B^\alpha_t, \]
 with $X_0$ uniform on $\cM$ and with \(B\) being a \(m\)-dimensional standard Brownian motion under some probability space \(\mathbb{P}\). The infinitesimal generator of $(X_t)_{t\ge 0}$ is $\cL_p$ by \eqref{eq:EDS-L}.
	By \Cref{proposition: Girsanov's theorem for embedded manifold}, $(X_t)_{t\ge 0}$ is also a solution to the SDE: 
	\[ \mathrm{d} X_t= \nabla\ln q(X_t)\mathrm{d}t + \sqrt{2}\sum_{\alpha=1}^m P_\alpha \circ \mathrm{d}\widetilde{B}^\alpha_t, \]
 with $X_0$ uniform on $\cM$, and 
	where $\widetilde B$ is a \(m\)-dimensional standard Brownian motion under the probability measure \(\mathbb{Q}\) defined by \(\frac{ \mathrm{d}\mathbb{Q}}{\mathrm{d} \mathbb{P}}= \mathcal{E}_T\) with 
	\[ \mathcal{E}_T=\exp\left( -M_T+\frac{1}{2} \langle M \rangle_T\right), \quad M_T=\int_0^T \frac{1}{\sqrt{2}}\left( \nabla \ln p - \nabla \ln q\right)(X_t)\mathrm{d}B_t. \]
	Therefore, by the definition of Kullback-Leibler divergence,
	\[ \mathrm{KL}(\P_{(p,T)}||\P_{(q,T)})=\mathbb{E}\SBRA{ \ln \frac{\mathrm{d} \mathbb{Q} }{\mathrm{d} \mathbb{P} } } = \mathbb{E}\SBRA{ -M_T+\frac{1}{2}\langle M \rangle_T }= \frac{1}{2}\mathbb{E}\SBRA{\langle M\rangle_T}. \]
	We have the desired conclusion by computing $\mathbb{E}\SBRA{\langle M\rangle_T}$, because the distribution of $X_t$ is $\mu=p\dd x$ at each time $t\ge 0$. 
	\end{proof}

{\footnotesize
\newcommand{\etalchar}[1]{$^{#1}$}

}

\end{document}